\pdfoutput=1
\PassOptionsToPackage{usenames,dvipsnames}{xcolor}
\documentclass[12pt, reqno]{amsart}
\usepackage{amsmath}
\usepackage{amssymb}
\usepackage{epsfig}
\usepackage{graphicx}
\usepackage{color}
\usepackage{fullpage}
\usepackage{comment}
\definecolor{shadecolor}{gray}{0.875}
\usepackage{amscd}
\usepackage{stmaryrd}

\usepackage{ulem}
\usepackage{amsthm}

\usepackage{csquotes}
\usepackage{enumitem}
\RequirePackage{tikz-cd}

\usepackage[usenames,dvipsnames]{xcolor}
\usepackage[pagebackref]{hyperref}
\hypersetup{
	colorlinks=true,
	citecolor=OliveGreen,
	linktocpage,
}
\RequirePackage{doi}
\RequirePackage{url}

\let\oldtocsection=\tocsection

\let\oldtocsubsection=\tocsubsection

\let\oldtocsubsubsection=\tocsubsubsection

\renewcommand{\tocsection}[2]{\hspace{0em}\oldtocsection{#1}{#2}}
\renewcommand{\tocsubsection}[2]{\hspace{1em}\oldtocsubsection{#1}{#2}}
\renewcommand{\tocsubsubsection}[2]{\hspace{2em}\oldtocsubsubsection{#1}{#2}}
\makeatletter

\newcommand{\Rmnum}[1]{\expandafter\@slowromancap\romannumeral #1@}
\makeatother
\numberwithin{equation}{section}

\newcommand{\ie}{\textit{i}.\textit{e}.}

\newcommand{\resp}{resp.\,}

\input xy
\xyoption{all}

\calclayout
\allowdisplaybreaks[3]

\theoremstyle{plain}
\newtheorem{prop}{Proposition}[section]

\newtheorem{theo}[prop]{Theorem}
\newtheorem{coro}[prop]{Corollary}

\newtheorem{lemm}[prop]{Lemma}

\theoremstyle{definition}
\newtheorem{defi}[prop]{Definition}

\newtheorem{conj}[prop]{Conjecture}

\newtheorem{rema}[prop]{Remark}

\newtheorem{exam}[prop]{Example}

\newtheorem{nota}[prop]{Notation}

\def\ra{\rightarrow}

\def\cF{{\mathcal F}}

\def\cL{{\mathcal L}}

\def\cO{{\mathcal O}}

\def\cU{{\mathcal U}}

\def\bA{{\mathbb A}}

\def\bC{{\mathbb C}}

\def\bP{{\mathbb P}}
\def\bQ{{\mathbb Q}}
\def\bR{{\mathbb R}}
\def\bZ{{\mathbb Z}}

\def\Br{\mathrm{Br}}

\def\Eff{\overline{\mathrm{Eff}}}
\def\Pic{\mathrm{Pic}}

\def\et{\mathrm{\acute{e}t}}

\def\vol{\mathrm{Vol}}

\def\Nef{\mathrm{Nef}}
\def\Hilb{\mathrm{Hilb}}
\def\Supp{\mathrm{Supp}}

\def\Spec{\mathrm{Spec}}

\def\Pic{\mathrm{Pic}}

\def\Sing{\mathrm{Sing}}
\def\Span{\mathrm{Span}}
\def\Gal{\mathrm{Gal}}

\def\Proj{\mathrm{Proj}}

\def\piet{\pi_1^{\text{\'et}}}

\makeatother
\makeatletter

\author{Runxuan Gao}
\address{Graduate School of Mathematics, Nagoya University, Furocho Chikusa-ku, Nagoya, 464-8602, Japan}
\email{m20015x@math.nagoya-u.ac.jp}

\title[Examples]{The geometric exceptional set in Manin's conjecture for Batyrev and Tschinkel's example}

\begin{document}

\date{\today}

\begin{abstract}
%%Abstract for presenting
%Manin's conjecture gives a prediction on how birational invariants control the counting function of rational points on algebraic varieties. Typically, there are subsets where rational points accumulate, and these need to be excluded from the counting. Batyrev and Tschinkel’s example demonstrates that merely removing a proper closed subset is insufficient. Nevertheless, in all known cases, it suffices to remove a thin exceptional set.
%
%Recently, Lehmann, Sengupta, and Tanimoto proposed an explicit geometric description of the thin exceptional set in Manin's conjecture, which we refer to as the 'geometric exceptional set'. In this study, we explicitly compute the geometric exceptional set for Batyrev and Tschinkel’s example.
%Lehmann, Sengupta, and Tanimoto introduced a conjectural construction of the thin exceptional set in Manin's conjecture, which we call the geometric exceptional set.
Batyrev and Tschinkel's example is a Fermat cubic surface bundle $X$ which is a Fano $5$-fold.
It is the first example for which Manin's conjecture can never hold for a proper closed exceptional set.
Recently, Lehmann, Sengupta, and Tanimoto proposed a conjectural geometric description of the exceptional set in Manin's conjecture and showed that it is always contained in a thin set.
Over a field of characteristic $0$, we explicitly construct finitely many thin maps such that any thin map $f:Y\rightarrow X$ with equal or larger $a$- and $b$-values in lexicographical order factors rationally through one of them.
In particular, this defines a thin set which coincides with Lehmann-Sengupta-Tanimoto's conjectural exceptional set.
\end{abstract}
\maketitle
\tableofcontents
%\newpage

\section{Introduction}
One of the major goals in algebraic geometry is to study rational points on varieties.
They are, roughly speaking, solutions to a system of polynomial equations.
It is believed that the behavior of rational points reflects the geometry of the variety, and vice versa.
For instance, Lang's conjecture \cite{Lang1983} predicts that the set of rational points on a variety of general type is not Zariski dense.
In dimension $1$, this is Faltings' celebrated theorem \cite{Faltings1983,Faltings1991}.

On the other extreme, a conjecture of Colliot-Th\'el\`ene \cite[p. 174]{ColliotThlne2003} implies that the set of rational points on a smooth projective geometrically rationally connected variety defined over a number field is Zariski dense as soon as it is not empty.
In this case, it is natural to study the distribution of rational points with respect to some heights.
Manin's conjecture predicts that the asymptotic formula for the number of rational points of bounded height on this kind of varieties depends only on geometric invariants of the underlying variety.

It turns out that on many varieties,
some proper closed subset contains all the rational points asymptotically.
%the proportion of rational points contained in some proper closed subsets converges to $1$ as the height tends to infinity.
%there is some proper closed subset possessing 100\% of the rational points as the bound tends to infinity, such as lines on del Pezzo surfaces
Thus we must remove a subset of rational points to count, called the exceptional set, to obtain a meaningful asymptotic formula reflecting the geometry of the whole variety.
(In fact, there are other types of rational points that should also be excluded, see Remark \ref{rema-ges}.)

The original Manin's conjecture, formulated in \cite{FMT89,Batyrev1990}, predicts that the exceptional set is always a proper closed subset, and this version of the conjecture has been established for many classes of varieties: generalized flag varieties \cite{FMT89}, toric varieties \cite{batyrev1998manin}, low-degree hypersurfaces of high dimensions \cite{FMT89,Sch79,Peyre95,Hooley1994}, and so forth.

However, counterexamples to this closed-set version of Manin's conjecture were found by Batyrev and Tschinkel \cite{batyrev1996rational} (see Propsition \ref{prop-BT}).
As such, it was suggested in \cite{Peyre03} to relax the exceptional set to be a thin set in the sense of Serre \cite{serre2016topics} (see Definition \ref{defi-thin-set}).
%and there is no known counterexample to the thin-set version of Manin's conjecture.
This makes sense since Colliot-Th\'el\`ene's conjecture \cite[p. 174]{ColliotThlne2003} also implies that the set of rational points is not thin as soon as it is not empty.

While Manin's conjecture has an explicit asymptotic formula, it lacked a geometric description of the thin exceptional set until a conjectural construction, which we call the geometric exceptional set to eliminate ambiguity, was proposed in \cite{LTDuke,LST} (see Definition \ref{defi-ges}).
In \cite{LTDuke,LST}, the authors proved that the geometric exceptional set for geometrically uniruled smooth projective variety is always contained in a thin set,
which provides evidence for both the thin-set version of Manin's conjecture and the geometric exceptional set.
The idea of the geometric exceptional set is straightforward, but its computation is not easy in general.
In the present paper, we compute the geometric exceptional set for Batyrev and Tschinkel's example.

%The original paper \cite{batyrev1996rational} of Batyrev and Tschinkel's example did not provide a comprehensive description of the rational points that should be included in the exceptional set.
%To the best of the author's knowledge, no literature has ever precisely provided which points should be included in the exceptional set in this case.

Let $X$ be a smooth projective variety over a number field $F$ and $\cL$ be an adelically metrized bundle on $X$,
where we denote the underlying divisor of $\cL$ by $L$.
Then we may define a height function
$$H_\cL:X(F)\rightarrow\bR,$$
see \cite[Section 2]{chambert2010igusa} for the basics on heights and adelically metrized line bundles.
%If $\cL_1$ and $\cL_2$ are two adelically metrized line bundles defined over the same line bundle $L$, then the ratio $H_{\cL_1}(x)/H_{\cL_2}(x)$ is bounded on $X$ (see \eg \cite{Peyre03}).
%When $L$ is ample, $H_\cL$ has the Northcott property that the counting function
%$$N(U,\cL,B):=\#\{ x\in U\mid H_\cL(x)\leq B \}$$
%is always finite for any subset $U\subseteq X(F)$.
When $L$ is ample, the counting function
$$N(U,\cL,B):=\#\{ x\in U\mid H_\cL(x)\leq B \}$$
takes finite values for any subset $U$ of $X$ and when $L$ is big and nef, it is finite for any $U$ outside a proper closed subset of $X$, see \cite[Section 2.2]{LT19}.

Recall that a thin set of $X(F)$ is a subset of $X(F)$ which is a finite union $\bigcup_if_i(Y_i(F))$ where $f_i:Y_i\rightarrow X$ are generically finite morphisms of varieties which admit no rational section (\ie\ a rational point on the generic fiber), see Section \ref{sec-thin-set}.
Manin's conjecture predicts that the rational point counts are asymptotically determined by geometric invariants of $X$ and $L$.
The following version of Manin's conjecture appears in \cite{LST}.

%The thin-set version of Manin's conjecture is as follows.
\begin{conj}[Manin's conjecture]\label{conj-manin}
	Let $X$ be a smooth projective geometrically rationally connected variety over a number field $F$ and $\cL$ be an adelically metrized big and nef line bundle on $X$.
	Suppose that $X(F)$ is not a thin set. Then there exists a thin subset $Z\subset X(F)$ such that
	\begin{equation}\label{eq-manin}
		N(X(F)\backslash Z,\cL,B)\sim c(F,Z,\cL)B^{a(X,L)}\log(B)^{b(F,X,L)-1}
	\end{equation}
	as $B\rightarrow\infty$.
\end{conj}
\begin{rema}
	The invariants $a(X,L)$ and $b(F,X,L)$ will be defined in Section \ref{ab} and $c(F,Z,L)$ is the Peyre's constant introduced in \cite{Peyre95,BT98}.
	When $(X,L)$ is adjoint rigid (see Section \ref{sec-adjoint-rigid}), $c(F,Z,L)$ is independent of the exceptional set $Z$ by definition.
\end{rema}
%\begin{rema}
%	The thin-set version of Manin's conjecture is compatible with the closed-set version in the following sense.
%	Suppose $X$ is a Fano variety with $L=-K_X$ and (\ref{eq-manin}) holds for a proper closed subset $Z\subset X(F)$.
%	Then it was shown in \cite[Theorem 1.2]{browning2019sieving} that a thin set of $X(F)$ contributes only 0\% to the counting function $N(X(F)\backslash Z,\cL,B)$ as $B$ tends to infinity.
%\end{rema}

The key special case of Batyrev and Tschinkel's example is as follows.
\begin{nota}[Fermat cubic surface bundle]\label{notation}
	Let $X$ be the smooth Fano hypersurface in $\bP_x^3\times\bP_y^3$ defined by the equation
	$$x_0 y_0^3+x_1 y_1^3+x_2 y_2^3+x_3 y_3^3=0$$
	over a field $F$ of characteristic $0$ and let $L=-K_X$ be the anticanonical divisor.
	Write $\pi_x:X\ra \bP_x^3$ and $\pi_y:X\ra \bP_y^3$ for the canonical projections.
\end{nota}
\begin{prop}[\cite{batyrev1996rational,FLS18}]\label{prop-BT}
	With Notation \ref{notation}, Manin's conjecture for $(X,L)$ is not true for any proper closed exceptional set.
\end{prop}

\begin{rema}\label{rema1.7}
	%	We may call the variety $X$ in Example \ref{exam-BT} the Fermat cubic surface bundle.
	In the original paper \cite{batyrev1996rational}, the coefficients $x_i$ are replaced by linear equations in two or more variables,
	so that the example can be generalized to any dimension $\geq3$.
	%	The assumption that $F$ contains $\bQ(\sqrt{-3})$ is to obtain a lower bound of the growing rate of rational points on smooth cubic surfaces
	The assumption that $F$ contains $\bQ(\sqrt{-3})$ in \cite{batyrev1996rational} can be removed by \cite{FLS18}.
\end{rema}

We propose a construction of the thin exceptional set for Batyrev and Tschinkel's example.
\begin{nota}[constructing a thin exceptional set]\label{nota-thin-set}
	We follow Notation \ref{notation}.
	For each $\tau\in\mathfrak{S}_4$, define the hypersurface $T_\tau$ in $\bP_x^3\times\bP_{s,t}^1$ by the equation
	$$ s^3x_{\tau(0)}x_{\tau(1)}-t^3x_{\tau(2)}x_{\tau(3)}=0.$$
	(In fact, it is sufficient to pick three such $\tau$, see Section \ref{sec-proof-of-thm}.)
	%	Then the canonical projection $g_\tau:T_\tau\ra \bP_x^3$ is a finite morphism of degree $3$.
	We define $f_\tau:X_\tau:=X\times_{\bP_x^3}T_\tau\rightarrow X$ by the base change along 
	$\pi_x$.
	%	Then $f_\tau$ is finite of degree $3$ as well since $\pi_1$ is an algebraic closed field extension of function fields.
	For each $\tau\in\mathfrak{S}_4$, define the subvariety $V_\tau\subset X$ by the equations
	$$x_{\tau(0)}y_{\tau(0)}^3+x_{\tau(1)}y_{\tau(1)}^3=x_{\tau(2)}y_{\tau(2)}^3+x_{\tau(3)}y_{\tau(3)}^3=0.$$
\end{nota}
\begin{rema}\label{rema-algebraic-fiber-space}
	The morphisms $f_\tau:X_\tau\ra X$ are thin maps since they are finite.
	Moreover, $X_\tau$ is geometrically irreducible since the tensor product $\overline{F}(\overline{X})\otimes_{\overline{F}(\overline{\bP_x^3})}\overline{F}(\overline{T_\tau})$ of function fields has only one minimal prime ideal by \cite[Proposition B.98]{ulrich}, where we used the fact that $\overline{F}(\overline{\bP_x^3})$ is algebraically closed in $\overline{F}(\overline{X})$ since $X\ra\bP_x^3$ is an algebraic fiber space, see \cite[Example 2.1.12]{Lazarsfeld2004}.
\end{rema}

The following is the main theorem of this paper.
\begin{theo}
	\label{theo-factor-through}
	We follow Notation \ref{notation} and Notation \ref{nota-thin-set}, where $F$ is a field of characteristic $0$.
	Then for any thin map $f:Y\rightarrow X$ from a smooth geometrically integral variety $Y$ defined over $F$ such that
	\begin{equation}\label{eq-ab}
		(a(Y,f^\ast L),b(F,Y,f^\ast L))\geq(a(X,L),b(F,X,L))=(1,2),
	\end{equation}
	there exists $\tau\in \mathfrak{S}_4$ such that $f$ factors rationally though either $V_\tau\subset X$ or $f_\tau:X_\tau\rightarrow X$.
\end{theo}

Conversely, $V_\tau\subset X$ and $f_\tau:X_\tau\rightarrow X$ are thin maps satisfying (\ref{eq-ab}).
The following corollary follows immediately by Theorem \ref{theo-factor-through} and Lang-Nishimura theorem.
\begin{coro}\label{coro-thin-set}
	We follow Notation \ref{notation} and Notation \ref{nota-thin-set}, where $F$ is a field of characteristic $0$.
	Then the subset
	$$Z:=\bigcup_{\tau\in\mathfrak{S}_4}V_\tau(F)\cup \bigcup_{\tau\in\mathfrak{S}_4} f_\tau(X_\tau(F))$$
	of $X(F)$ is a thin set,
	and we have that
	\begin{equation}\label{eq}
		Z=\bigcup_f f(Y(F))
	\end{equation}
	where $f:Y\rightarrow X$ varies over all $F$-thin maps where $Y$ is a smooth geometrically integral variety satisfying
	$$(a(Y,f^\ast L),b(F,Y,f^\ast L))\geq(a(X,L),b(F,X,L))=(1,2)$$
	in lexicographical order.
	In particular, the thin set $Z$ coincides with the geometric exceptional set defined in Definition \ref{defi-ges}.
\end{coro}

Conversely, any rational point on $V_\tau$ lies on a line $l$ on a smooth $\pi_x$-fiber, which satisfies $a(l,L|_l)=2$;
any rational point in $f_\tau(X_\tau(F))$ lies on a smooth $\pi_x$-fiber $Y$ with $a(Y,L|_Y)=1$ and $b(F,Y,L|_Y)\geq 2$.

With Corollary \ref{coro-thin-set}, we can propose a precise statement of Manin's conjecture for $X$, which is a restatement of \cite[Conjecture 5.1]{LST}.

\begin{conj}[Manin's conjecture for the diagonal cubic surface bundle]\label{conj-manin-cubic}
	Let $(X,L)$ be defined as in Notation \ref{notation} where we assume that $F$ is a number field.
	Then Conjecture \ref{conj-manin} is true for $(X,L)$ with the exceptional set $Z$ defined in Corollary \ref{coro-thin-set}.
\end{conj}
%\begin{rema}
%	Suppose Manin's conjecture is true for smooth diagonal cubic surfaces.
%	By Remark \ref{rema-conversely}, if Conjecture \ref{conj-manin-cubic} is true,
%	then the thin set $Z$ is the minimal one, up to a modification of a thin set, to make it happen.
%\end{rema}

The structure of the paper is as follows.
In Section \ref{sec-preliminaries}, we review the background on birational geometry relevant to Manin's conjecture,
where the geometric exceptional is defined in Section \ref{sec-ges}.
Section \ref{sec-the-example} is devoted to studying the geometric exceptional set for Batyrev and Tschinkel's example.
The proofs of Theorem \ref{theo-factor-through} and Corollary \ref{coro-thin-set} are presented in Section \ref{sec-proof-of-thm}.

\subsection*{Comparison to previous works}

The geometric exceptional set for (smooth) del Pezzo surfaces is almost well-understood.
It was shown in \cite{LT19} that a del Pezzo surface of degree $\geq2$ or a general del Pezzo surface of degree $1$ have a proper closed geometric exceptional set.
Nevertheless, examples of del Pezzo surfaces of degree $1$ with Zariski dense geometric exceptional set have been discovered in \cite{gao2023zariski},
which provide the first classes of counterexamples to the closed-set version of Manin's conjecture in dimension $2$.

When comes to higher dimensions, Batyrev and Tschinkel's cubic surface bundles \cite{batyrev1996rational} provide such counterexamples in any dimension $\geq3$.
In \cite[Section 12]{LT19}, the geometric exceptional set for the diagonal quadric surface bundle has been computed, for which the 
Manin's conjecture has been established in \cite{BTB20} with the same exceptional set.
Manin's conjecture for the diagonal cubic surface bundle is still open.
Notably, the conjecture is open for even a single smooth cubic surface.

\noindent
{\bf Acknowledgments.}
The author thanks his advisor, Sho Tanimoto, for suggesting the problem and for many useful conversations.
The author also thanks anonymous referees for careful reading and helpful suggestions.
The author was partially supported by JST FOREST program Grant number JPMJFR212Z and JSPS Bilateral Joint Research Projects Grant number JPJSBP120219935 and Grant-in-Aid for JSPS Fellows Number 24KJ1234.

\subsection*{Notations and Conventions}
A variety $X$ is a separated scheme of finite type over a field $F$.
In this paper, we assume that $F$ is of characteristic $0$.
We denote by $\overline{X}$ the base change to a fixed algebraic closure $\overline{F}$.
A subvariety $Y$ of $X$ is a subscheme which is a variety.
When $X$ is projective, we sometimes use subvariety to indicate projective subvariety.

%When we say a $\bQ$-divisor, we basically assume it is $\bQ$-Cartier.

For real numbers $a_1,b_1,a_2,b_2$, we write $(a_1,b_1)>(a_2,b_2)$ for the lexicographical order, that is, $(a_1,b_1)>(a_2,b_2)$ if and only if either $a_1>a_2$, or $a_1=a_2$ and $b_1>b_2$. The notation $(a_1,b_1)=(a_2,b_2)$ means that $a_1=a_2$ and $b_1=b_2$. And we write $(a_1,b_1)\geq(a_2,b_2)$ when either $(a_1,b_1)>(a_2,b_2)$ or $(a_1,b_1)=(a_2,b_2)$.
For two functions $f,g:\bR\rightarrow\bR$, we write $f(B)\sim g(B)$ to indicate $\lim_{B\rightarrow\infty}f(B)/g(B)=1$.

\section{Preliminaries}\label{sec-preliminaries}

\subsection{Thin maps and thin sets}\label{sec-thin-set}
A morphism $f:Y\ra X$ between varieties is called \textit{generically finite} if the preimage of the generic point of any irreducible component of $X$ is a finite set.
Note that we do not require $f$ to be dominant in the above definition.
The notion of thin sets appears in \cite{serre2016topics}.
\begin{defi}\label{defi-thin-set}
	Let $f:Y\rightarrow X$ be a generically finite morphism of projective varieties over a field $F$.
	Then $f$ is called a \textit{thin map} if it does not admit any rational section (\ie\ a rational point on the generic fiber).
	%	\begin{enumerate}
		%		\item $f$ is not dominant, or
		%		\item $f$ is dominant with $\deg f\geq2$.
		%	\end{enumerate}
	For a projective variety $X$ over a field $F$, a subset $Z\subset X(F)$ is called a \textit{thin set} if $Z$ is a finite union of $f(Y(F))$ where $f:Y\rightarrow X$ is a thin map.
\end{defi}
\begin{rema}\label{rema-surjective}
	%	Since the varieties in the definition are projective, the morphism $f$ is surjective onto its image.
	%	In Definition \ref{defi-thin-set}, merely assuming that $Y$ is quasi-projective does not affect 
	There are two types of a thin map $f:Y\rightarrow X$. When $f$ is dominant, we must have $\deg f\geq2$; whereas when $f$ is not dominant, it can either have degree $\geq2$ or be birational onto its image.
\end{rema}

\subsection{$a$- and  $b$-invariants in Manin's conjecture}\label{ab}
%\subsection{the $a$-invariant}
We quickly review some basic concepts in birational geometry, see \cite{Lazarsfeld2004} for more details.
Let $X$ be a projective variety.
Then we can define an intersection pairing between Cartier divisors and $1$-cycles.
We denote by $N^1(X)_\bZ$ (\resp $N_1(X)_\bZ$) the group of Cartier divisors (\resp $1$-cycles) quotiented by those which have zero intersection with any $1$-cycles (\resp effective Cartier divisors).
We define
$$N^1(X):=N^1(X)_\bZ\otimes\bR,\quad N_1(X):=N_1(X)_\bZ\otimes\bR$$
and attach them the Euclidean topology as $\bR$-vector spaces.

A subset $C$ of an $\bR$-vector space is called a \textit{cone} if $\bR_{\geq0}\cdot C\subseteq C$.
A Cartier divisor (\resp $1$-cycle) is called \textit{nef} if it has non-negative intersection with any effective $1$-cycle (\resp Cartier divisor).
After tensoring by $\bR$, nef divisors (\resp $1$-cycle) form a convex cone denoted by $\Nef^1(X)$ (\resp $\Nef_1(X)$).
We define the \textit{pseudo-effective cone} $\Eff^1(X)$ to be the closure of the convex cone spanned by effective $\bR$-divisors in $N^1(X)$.
It turns out that the $\Eff^1(X)$ and $\Nef_1(X)$ are dual in the sense that
$$\Nef_1(X)=\left\{ Z\in N_1(X) \mid  D\cdot Z\geq 0 \text{ for any }D\in\Eff^1(X) \right\}.$$
A Cartier divisor $D$ is called \textit{big} if $mD\equiv A+E$ for some positive integer $m$, ample divisor $A$, and effective divisor $E$.

Let $L$ be a line bundle on smooth projective variety $X$.
If $H^0(X,mL)=0$ for all $m>0$, we define $\kappa(X,L)=-\infty$; otherwise, we define $\kappa(X,L)$ to be the maximum of $\dim (\phi_m(X))$ for $m>0$, where $\phi_m$ is the rational map $X\dasharrow \bP H^0(X,mL)$ associated to the linear system $\lvert mL\rvert$.
We call $\kappa(X,L)$ the \textit{Iitaka dimension} of $(X,L)$.
%When $X$ is singular, the Iitaka dimension of $(X,L)$ is defined to be the Iitaka dimension of $(\widetilde{X},\beta^\ast L)$ for any smooth resolution $\beta:\widetilde{X}\ra X$.

The $a$-invariant, also known as the Fujita invariant or the pseudo-effective threshold, and its negative, referred to as the Kodaira energy, is defined as follows.
\begin{defi}\label{defi-a}
	Let $X$ be a projective variety over a field of characteristic $0$ and $L$ be a big and nef $\bQ$-divisor on $X$.
	If $X$ is smooth, the $a$-invariant is defined as
	$$a(X,L):=\inf\{ t\in\bR\mid K_X+tL\in \Eff^1(X)\}.$$
	If $X$ is not smooth, then take a smooth resolution $\beta:\widetilde{X}\rightarrow X$ and define $a(X,L):=a(\widetilde{X},\beta^\ast L)$.
\end{defi}
\begin{rema}
	By Proposition \ref{prop-propoties-of-a}, the definition is independent of $\beta$, and the $a$-invariant is a birational invariant.
\end{rema}

\begin{exam}
	When $X$ is a smooth Fano variety and $L=-K_X$, we have that $a(X,L)=1$.
\end{exam}

The following is one of the main results in \cite{BCHM10}.
\begin{prop}\label{prop-a}
	Let $X$ be a smooth projective variety and $L$ be a big and nef $\bQ$-divisor on $X$.
	If $a(X,L)>0$, then $a(X,L)$ is a rational number.
	In this case, we can characterize the $a$-invariant by
	$$a(X,L)=\min\left\{ \frac{r}{s}\in\bQ \middle| h^0(X,sK_X+rL)>0 \right\}.$$
\end{prop}
\begin{proof}
	For the first assertion, see \cite[Corollary 1.1.7]{BCHM10} when $L$ is ample
	and \cite[Theorem 2.16]{HTT15} when $L$ is big and nef.
	%	 (see Remark \ref{rema-bchm} for an explanation of how \cite[Corollary 1.1.7]{BCHM10} is used).
	For the second assertion, see \cite[equation (3.2)]{LT19}.
\end{proof}

The $a$-invariant is a birational invariant and is preserved under field extension by the following proposition.
\begin{prop}\label{prop-propoties-of-a}
	Let $X$ be a projective variety defined over a field $F$ of characteristic zero and $L$ be a nef and big $\bQ$-divisor.
	\begin{enumerate}[ label={\rm(\arabic*)}]
		\item {(\rm\cite[Proposition 2.7]{HTT15})} Assume that $X$ is $\bQ$-factorial and has canonical singularities.
		Let $\beta:\widetilde{X}\rightarrow X$ be a birational morphism where $\widetilde{X}$ is smooth.
		Then $a(\widetilde{X},\beta^\ast L)=a(X,L)$.
		\item {\rm (\cite[Corollary 4.5]{LST})} Assume that $X$ is smooth.
		For any field extension $F^\prime/F$, let $X^\prime$ be any irreducible component of $X_{F^\prime}$.
		Then we have that $a(X^\prime,L)=a(X,L)$ and $\kappa(X^\prime,K_{X^\prime}+a(X^\prime, L)L)=\kappa(X,K_X+a(X, L)L)$.
	\end{enumerate}
\end{prop}

We list two geometric properties of the $a$-invariant which will be repeatedly used.
\begin{prop}\label{prop-a-dominant} {\rm (\cite[Lemma 4.7]{LST})} 
	Let $f:Y\rightarrow X$ be a dominant generically finite morphism of projective varieties and $L$ be a big and nef $\bQ$-divisor on $X$.
	Then $a(Y,f^\ast L)\leq a(X,L)$.
\end{prop}

See the beginning of Section \ref{sec-families} for the definition of a dominant family.
The first version of the following result was proved in \cite[Proposition 4.1]{LTT} where the base field was assumed to be algebraically closed.
\begin{prop}\label{prop-a-dominant-member}{\rm (\cite[Lemma 4.8]{LST})} 
	Let $X$ be a smooth projective variety and $L$ be a big and nef $\bQ$-divisor on $X$.
	Let $Y$ be a general member of a dominant family of subvarieties of $X$.
	Then $a(Y,L)\leq a(X,L)$.
\end{prop}

Varieties with large $a$-values have been classified in \cite{fujita1989remarks} (see also \cite[Table 7.1]{beltrametti2011adjunction} and \cite{liu2022note}).
\begin{prop}\label{prop-large-a} {\rm (\cite[Theorem 1.3]{liu2022note})}
	Let $X$ be a normal projective variety of dimension $n$ over a field of characteristic $0$ and let $L$ be an ample divisor on $X$.
	If $a(X,L)> n-1$, then either
	\begin{enumerate}[ label={\rm(\arabic*)}]
		\item $(X,L)\cong (\bP^n,\cO(1))$; or
		
		%		\item  $(X,L)\cong(Q,\cO(1))$, where $Q$ is a quadric hypersurface in $\bP^{n+1}$; or
		
		\item a smooth resolution of $(X,L)$ is isomorphic to a $\bP^{n-1}$-bundle over a smooth curve $C$ with polarization $\cO(1)$; or
		\item a smooth resolution of $(X,L)$ is isomorphic to $(\bP_{\bP^2}(\cO^{\oplus(n-2)}\oplus\cO(2)),\cO(1))$.
	\end{enumerate}
	%		\begin{enumerate}[ label={\rm(\arabic*)}]
		%		\item $a(X,L)=n+1$ in which case $(X,L)\cong (\bP^n,\cO(1))$ and $(X,L)$ is adjoint rigid; or
		%		
		%		\item  $a(X,L)=n$ and $(X,L)$ is adjoint rigid, in which case $(X,L)\cong(Q,\cO(1))$, where $Q$ is a quadric hypersurface in $\bP^{n+1}$; or
		%		
		%		\item $a(X,L)=n$ and $(X,L)$ is not adjoint rigid, in which case a resolution of singularities of $(X,L)$ is isomorphic to a $\bP^{n-1}$-bundle over a smooth curve $C$ with polarization $\cO(1)$; or
		%		\item $a(X,L)=n-\frac{1}{2}$ and $(X,L)$ is adjoint rigid, in which case a resolution of singularities of $(X,L)$ is isomorphic to $(\bP_{\bP^2}(\cO^{\oplus(n-2)}\oplus\cO(2)),\cO(1))$.
		%	\end{enumerate}
\end{prop}

The $b$-invariant is defined as follows.
\begin{defi}
	Let $X$ be a projective variety over a field $F$ of characteristic $0$ and let $L$ be a big and nef $\bQ$-divisor.
	If $X$ is smooth, we define
	$$\cF_{X,L}:=\{\alpha\in\Nef_1(X) \mid (K_X+a(X,L)L)\cdot\alpha=0\},$$
	and the $b$-invariant of $X$ is defined to be
	$$b(F,X,L):=\dim \cF_{X,L}.$$
	If $X$ is not smooth, we take a smooth resolution $\beta:\widetilde{X}\rightarrow X$ and define $b(F,X,L):=b(F,\widetilde{X},\beta^\ast L)$.
	We may simply denote $b(X,L)$ when $F$ is explicit.
\end{defi}

\begin{rema}\label{rema-b-is-birational}
	The definition is independent of $\beta$ by \cite[Propostion 2.10]{HTT15}.
	By the same Proposition, the $b$-invariant is a birational invariant.
	This definition of $b$-invariant is equivalent to the one in \cite[Definition 2.8]{HTT15}, see the paragraph after \cite[Definition 4.1]{LT19} for an explanation.
\end{rema}

\begin{exam}
	Let $X$ be a smooth Fano variety over a field $F$ and set $L=-K_X$. Then $b(F,X,L)$ equals to the Picard rank of $X$.
\end{exam}

The $b$-invariant does not enjoy so much geometric properties as the $a$-invariant, and it is not stable under field extension; but it can only increase.

\begin{prop}\label{prop-properties-of-b}{\rm (\cite[Proposition 4.13]{LST})}
	Let $X$ be a geometrically integral smooth projective variety defined over a field $F$ of characteristic $0$ and $L$ be a big and nef $\bQ$-divisor on $X$.
	Then for any finite field extension $F^\prime /F$, we have that
	$$b(F,X,L)\leq b(F^\prime,X_{F^\prime},L_{F^\prime}).$$
\end{prop}

The notion of face-contracting morphisms, first defined in \cite[Definition 3.6]{GMC}, is crucial for ensuring that the geometric exceptional set is contained in a thin set.
In this paper, this notion appears only in the definition of the geometric exceptional set (Definition \ref{defi-ges}).
\begin{defi}
	Let $X,Y$ be smooth projective varieties and $L$ be a big and nef $\bQ$-divisor on $X$.
	Let $f:Y\rightarrow X$ be a dominant generically finite morphism with $a(Y,f^\ast L)=a(X,L)$.
	Then the pushforward $f_\ast: N_1(Y)\rightarrow N_1(X)$ maps $\cF_{Y,f^\ast L}$ into $\cF_{X,L}$ \cite[Lemma 4.25]{LST}, and we say that $f$ is a \textit{face-contracting} morphism if $f_\ast:\cF_{Y,f^\ast L}\rightarrow \cF_{X,L}$ is not injective.
\end{defi}

\subsection{Adjoint rigid varieties}\label{sec-adjoint-rigid}

Let $X$ be a smooth projective variety over a field $F$ of characteristic $0$ and $L$ be a big and nef $\bQ$-divisor.
Let $d$ be a positive integer such that $d\Delta$ is integral for the divisor $\Delta$ mentioned in Remark \ref{rema-bchm}.
Then we can define a graded ring
$$R(X,K_X+a(X,L)L):=\bigoplus_{m\geq0}H^0(X,md(K_X+a(X,L)L)),$$
which is known as the section ring associated to $(X,a(X,L)L)$.
By \cite[Theorem 1.2]{BCHM10}, the ring $R(X,K_X+a(X,L)L)$ is a finite generated $F$-algebra (see Remark \ref{rema-bchm}), and thus defines a rational map
$$\phi:X\dasharrow Z:=\Proj\,R(X,K_X+a(X,L)L).$$
%which is known as the canonical model map for $(X,a(X,L)L)$.
We say the pair $(X,L)$ is \textit{adjoint rigid} if the image of $\phi$ is of dimension $0$.

\begin{rema}\label{rema-bchm}
	As in the proof of \cite[Theorem 2.3]{LTT}, there exists a $\bQ$-divisor $\Delta$ that is $\bQ$-equivalent to $a(X,L)L$, such that $(X,\Delta)$ is terminal and, therefore, klt.
	Then one can use \cite[Theorem 1.2]{BCHM10} to conclude that $R(X,K_X+a(X,L)L)$ is finitely generated.
\end{rema}

%\begin{defi}\label{defi-adjoint-rigid}
%	Let $X$ be a smooth projective variety with a big and nef $\bQ$-divisor $L$.
%	We say that the pair $(X,L)$ is adjoint rigid if $\dim H^0(X,m(K_X+a(X,L)L))=1$ for sufficiently divisible positive integer $m$.
%	%	We may simply say that $X$ is adjoint rigid if no confusion can arise with $L$.
%	We may simply say that $X$ is adjoint rigid when $L$ is explicit.
%	When $X$ is singular, we take a smooth resolution $\beta:\widetilde{X}\rightarrow X$ and use $(\widetilde{X},\beta^\ast L)$ to define the adjoint rigidness of $(X,L)$.
%\end{defi}
\begin{rema}\label{rema-rationally-conn}
	It follows from \cite{HM07} that a geometrically integral adjoint rigid variety is geometrically rationally connected.
\end{rema}
%In other words, $(X,L)$ is adjoint rigid if and only if the Iitaka dimension $\kappa(X,K_X+a(X,L)L)=0$.
%The following proposition says that a non-adjoint-rigid variety
%can be birationally fibered into adjoint rigid subvarieties with the same $a$-value.

The notion of adjoint rigidity is useful in the study of geometric exceptional set.
Roughly speaking, when studying the $a$- and $b$-invariants of a variety, we can reduce the problem to studying its adjoint rigid subvarieties, by the following proposition.
%\begin{prop}\label{prop-adjoint-rigid}{\rm (\cite[Lemma 4.6]{LTDuke}, \cite[Lemma 4.3]{LST})} 
\begin{prop}\label{prop-adjoint-rigid}{\rm (\cite[Lemma 4.6]{LTDuke}, \cite[Lemma 4.3]{LST})} 
	Let $Y$ be a smooth uniruled projective variety defined over a field of characteristic $0$ and $L$ be a big and nef $\bQ$-divisor.
	Suppose that $(Y,L)$ is not adjoint rigid and
	let $g:Y\dashrightarrow Z$ be the rational map associated to $(Y, K_Y+a(Y, L)L)$.
	Let $E$ be the closure of a general fiber of $g$ in $Y$.
	Then
	$$a(E,L|_E)=a(Y,L),\quad b(\overline{F},E,L|_E)\geq b(\overline{F},Y,L),$$
	%		$a(E,L|_E)=a(Y,L),$
	and $(E,L|_E)$ is adjoint rigid.
\end{prop}
\begin{proof}
	%	The assertion over $\overline{F}$ follows from \cite[Lemma 4.6]{LTDuke}.
	%	The full assertion then follows from (2) of Proposition \ref{prop-propoties-of-a}.
	The conclusion on $a$-invariants and adjoint rigidity is proved in \cite[Lemma 4.3]{LST} over an arbitrary field of characteristic $0$, and the conclusion on $b$-invariants is proved in \cite[Lemma 4.6]{LTDuke} over an algebraically closed field of characteristic $0$.
\end{proof}
In the above proposition, the fiber $E$ is projective and connected.
By generic smoothness (\cite[Theorem 21.6.4]{vakil}, the general fiber $E$ is irreducible.
Thus the $a$- and $b$-invariants are well-defined for $E$.

Let $\phi$ be a rational map.
We say that a rational map $X^\prime\ra Z^\prime$ is \textit{birationally equivalent} to $\phi$ if there is a commutative diagram
% https://q.uiver.app/#q=WzAsNCxbMCwwLCJYXlxccHJpbWUiXSxbMCwxLCJaXlxccHJpbWUiXSxbMSwwLCJYIl0sWzEsMSwiWiJdLFswLDIsInUiLDAseyJzdHlsZSI6eyJib2R5Ijp7Im5hbWUiOiJkYXNoZWQifX19XSxbMiwzLCJcXHBoaSIsMCx7InN0eWxlIjp7ImJvZHkiOnsibmFtZSI6ImRhc2hlZCJ9fX1dLFsxLDMsInYiLDIseyJzdHlsZSI6eyJib2R5Ijp7Im5hbWUiOiJkYXNoZWQifX19XSxbMCwxLCIiLDIseyJzdHlsZSI6eyJib2R5Ijp7Im5hbWUiOiJkYXNoZWQifX19XV0=
\[\begin{tikzcd}
	{X^\prime} & X \\
	{Z^\prime} & Z
	\arrow["u", dashed, from=1-1, to=1-2]
	\arrow[dashed, from=1-1, to=2-1]
	\arrow["\phi", dashed, from=1-2, to=2-2]
	\arrow["v"', dashed, from=2-1, to=2-2]
\end{tikzcd}\]
for some birational rational maps $u$ and $v$.

Over an arbitrary field of characteristic $0$, the following proposition characterizes the $b$-invariant in the situation of Proposition \ref{prop-adjoint-rigid}.

\begin{prop}\label{prop-b-inv-and-monodromy}{\rm (\cite[Lemma 4.15]{LST})} 
	Let $X$ be a geometrically uniruled and geometrically integral smooth projective variety defined over a field of characteristic $0$ and $L$ be a big and nef $\bQ$-divisor on $X$.
	Let $\pi:X\dashrightarrow W$ be a rational map.
	Suppose that there exists a non-empty open subset $W^\circ\subset W$, such that
	\begin{enumerate}[ label={\rm(\arabic*)}]
		\item $W^\circ$ is smooth; and
		\item $f$ is defined and smooth over $W^\circ$; and
		\item $a(X_w,L|_{X_w})=a(X,L)$ for any closed point $w\in W^\circ$.
		\item $\pi$ is birationally equivalent to the rational map associated to $(X, K_X+a(X, L)L)$.
	\end{enumerate}
	Then for any closed point $w\in W^\circ$ and a geometric point $\overline{w}$ above $w$, we have that
	$$b(F,X,L)=\dim \left( \left(N^1(X_w)\cap N^1(\overline{X}_{\overline{w}})^{\pi_1^{\et}(\overline{W^\circ},\overline{w})} \right)/\Span (\{ E_i \}^r_{i=1})\right),$$
	where $\{ E_i \}^r_{i=1}$ is the finite set of irreducible divisors which  dominates $W$ under $\pi$ and which satisfy $K_X+a(X,L)L-c_i E_i\in \Eff^1(X)$ for some $c_i>0$.
\end{prop}

The following proposition says that in the situation of Proposition \ref{prop-adjoint-rigid}, the adjoint rigidity can be passed to the image.
\begin{prop}\label{prop-adjoint-rigid-2}{\rm (\cite[Lemma 4.9]{LST})} 
	Let $f:Y\rightarrow X$ be a generically finite surjection of smooth projective varieties over a field $F$ of characteristic $0$ and $L$ be a big and nef $\bQ$-divisor on $X$ with $a(Y,f^\ast L)=a(X,L)$.
	Let $S$ be the image of the closure of a general fiber of the rational map associated to $(X,K_X+a(X, f^\ast L)f^\ast L)$.
	%	Let $E$ be the closure of a general fiber for the canonical map for $(Y,a(X, f^\ast L)f^\ast L)$ and $S$ be its image in $X$.
	Then $(S,L|_S)$ is adjoint rigid with $a(S,L|_S)=a(X,L)$.
\end{prop}

%\begin{prop}\label{prop-adjoint-rigid-curve}
%	Let $C$ be a curve with a big and nef divisor $L$.
%	If $a(C,L)>0$, then $(C,L)$ is adjoint rigid.
%\end{prop}
%\begin{proof}
%	content...
%\end{proof}

\subsection{Families of varieties}\label{sec-families}
%Let $\pi:\cU\rightarrow W$ be a morphisms of varieties.
%We call $\pi:\cU\rightarrow W$ a family of varieties if its generic fiber is geometrically integral.
%A family of projective varieties is a family of varieties $\pi:\cU\rightarrow W$ where $W$ and $\pi$ are projective.
%Let $X$ be a projective variety.
%We call
%% https://q.uiver.app/#q=WzAsMyxbMCwwLCJcXGNVIl0sWzEsMCwiWCJdLFswLDEsIlciXSxbMCwxLCJzIl0sWzAsMiwiXFxwaSIsMl1d
%\[\begin{tikzcd}
	%	\cU & X \\
	%	W
	%	\arrow["s", from=1-1, to=1-2]
	%	\arrow["\pi"', from=1-1, to=2-1]
	%\end{tikzcd}\]
	%a family of (projective) subvariety of $X$ if $s$ the restriction of $s$ to each fiber of $\pi$ is a closed immersion.
	
	%Let $\Chow(X)$ denote the Chow variety of $X$.
	%Then by the universal property of $\Chow(X)$, any family of projective varieties as above factors through the universal family $\Univ(X)\rightarrow\Chow(X)$.
	
	A \textit{family of varieties} is a proper, flat morphism $p:\cU\rightarrow W$ of schemes such that the generic fiber of $p$ is geometrically irreducible,
	and $W$ is irreducible.
	Let $X$ be a projective variety.
	A \textit{family of projective subvarieties of $X$} is a family $p$ as above with the following commutative diagram
	% https://q.uiver.app/#q=WzAsNCxbMSwwLCJYXFx0aW1lcyBXIl0sWzIsMCwiWCJdLFsxLDEsIlciXSxbMCwwLCJcXGNVIl0sWzAsMl0sWzAsMV0sWzMsMCwiIiwwLHsic3R5bGUiOnsidGFpbCI6eyJuYW1lIjoiaG9vayIsInNpZGUiOiJ0b3AifX19XSxbMywyLCJwIiwyXV0=
	\[\begin{tikzcd}
		\cU & {X\times W} & X \\
		& W
		\arrow[hook, from=1-1, to=1-2]
		\arrow["p"', from=1-1, to=2-2]
		\arrow[from=1-2, to=1-3]
		\arrow[from=1-2, to=2-2]
	\end{tikzcd}\]
	with canonical projections such that the left horizontal arrow realizing $\cU$ as a closed subscheme of $X\times W$.
	We say the above family of projective subvarieties of $X$ is a \textit{dominant family} if the composition of the horizontal maps is dominant.
	
	%		Note that this definition of families is not the usual one in the definition of Hilbert schemes.
	%		The Hilbert scheme $\Hilb(X)$ is a scheme represents the functor
	%		$$\mathcal{H}ilb(X)(T):=\left\{
	%		\begin{array}{l}
		%			\text{closed subschemes } \cU\subset X\times T\text{ such that }
		%			\cU\rightarrow T\text{ is flat}
		%		\end{array}
	%		\right\},$$
	%		see \cite[Section 3.1]{Kollr2023} for details.
	%		The scheme $\Hilb(X)$ is separated since $X$ is separated.
	%		We have $\Hilb(X)=\coprod_P \Hilb_P(X)$, where $\Hilb_P(X)$ parametrizes subschemes with Hilbert polynomials $P$
	%		and $\Hilb_P(X)$ is projective.
	%		Any family of projective subvarieties of $X$ (in our sense) admits a morphism $W\rightarrow\Hilb_P(X)$ for some $P$
	%		whose scheme-theoretic image is a projective variety.

	The following proposition says that the $a$- and $b$-invariants are generically constant in a family of projective varieties over an algebraically closed field.
	
	\begin{prop}\label{prop-constant-in-family}
		{\rm (\cite[Theorem 4.3 and Proposition 4.4]{LTDuke}, \cite[Theorem 4.6]{LST})} 
		Let $p:\cU\rightarrow W$ be a family of projective varieties defined over an algebraic closed field of characteristic $0$ and $L$ be an $f$-big and $f$-nef $\bQ$-divisor on $\cU$. Then the following holds.
		\begin{enumerate}[label={\rm(\roman*)}]
			\item There exists a nonempty open subset $W^\circ$ of $W$ such that $a(\cU_t,L|_{\cU_t})$ and $\kappa(\cU_t,K_{\cU_t}+a(\cU_t,L|_{\cU_t})L|_{\cU_t})$ are constant with respect to $t\in W^\circ$.
			\item If a general fiber of $f$ is adjoint rigid, then there exists a nonempty open subset $W^\circ$ of $W$ such that $b(\cU_t,L|_{\cU_t})$ is constant with respect to $t\in W^\circ$.
		\end{enumerate}
		%	If further assume that $f$ is smooth and $L$ is $f$-big and $f$-semiample, then
		%	\begin{enumerate}[resume]
			%		\item $a(\cU_t,L|_{\cU_t})$  and $\kappa(\cU_t,K_{\cU_t}+a(\cU_t,L|_{\cU_t})L|_{\cU_t})$ are constant with respect to $t\in W$.
			%	\end{enumerate}
	\end{prop}
	%\begin{proof}
	%	The proof of \cite[Theorem 4.3 and Proposition 4.4]{LTDuke} does not rely on the reducedness of $W$.
	%	We may assume $W$ is irreducible since we only need to find an nonempty open subset of $W$ with the desired property.
	%	Let $n:W\rightarrow \bN$ be the function counting the numbers of geometrically irreducible components of fibres of $p$.
	%	By \cite[\href{https://stacks.math.columbia.edu/tag/055A}{Tag 055A}]{stacks-project}, there exists an nonempty open subset $W^\circ\subset W$ such that $n$ is constant on $W^\circ$.
	%	Then the proof of \cite[Theorem 4.3 and Proposition 4.4]{LTDuke} applied.
	%\end{proof}
	The first part of the above proposition is due to the invariance of log plurigenera proved in \cite{HMX13}.
	Working with families of varieties allows us to reduce the problem of identifying thin maps to $(X,L)$ with higher geometric invariants to identifying families of subvarieties with higher geometric invariants, where the later ones form a bounded family when $L$ is semiample by \cite{Hacon2015}.
	
	Combining above arguments, we obtain the following corollary of \cite[Theorem 4.19]{LST}.
	%		and will be a key ingredient in the proof of Theorem \ref{theo-exceptional-set}.
	\begin{coro}\label{coro-family-breaking}
		Let $X$ be a geometrically uniruled smooth projective variety over a field of characteristic $0$ with a big and semiample $\bQ$-divisor $L$. Then there exist a proper closed subset $V\subset X$ and finitely many families $(s_i:\cU_i\rightarrow X,p_i:\cU_i\rightarrow W_i)$ $(i\in I)$ where $W_i$ is a closed subvariety of $\Hilb(X)$, such that the following assertions hold.
		\begin{enumerate}[ label={\rm(\arabic*)}]
			%		\item for each $i$, the evaluation morphism $s_i$ is generically finite; and
			\item For each $i$, a general fiber $E$ of $p_i$ is an adjoint rigid subvariety of $X$ with $a(E,L|_E)= a(X,L)$.
			\item Let $f:Y\rightarrow X$ be a generically finite morphism of smooth projective varieties with $a(Y,f^\ast L)=a(Y^\prime,L)=a(X,L)$, where $Y^\prime$ denotes the image of $f$.
			%				and $(Y,f^\ast L)$ is not adjoint rigid,
			%				
			Let $Y\dashrightarrow W$ be the rational map associated to $(Y,K_Y+a(Y,f^\ast L)f^\ast L)$.
			%				 and
			%				$\Gamma\subset Y\times W$ be its graph.
			Then there exist $i\in I$ and rational maps $g$ and $h$ such that the following diagram commutes.
			% https://q.uiver.app/#q=WzAsNSxbMSwwLCJcXGNVX2kiXSxbMSwxLCJXX2kiXSxbMiwwLCJYIl0sWzAsMSwiVyJdLFswLDAsIlkiXSxbMCwyLCJzX2kiXSxbMCwxLCJwX2kiLDIseyJsYWJlbF9wb3NpdGlvbiI6MjB9XSxbMywxLCJnIiwwLHsic3R5bGUiOnsiYm9keSI6eyJuYW1lIjoiZGFzaGVkIn19fV0sWzQsMiwiZiIsMCx7ImxhYmVsX3Bvc2l0aW9uIjoyMH1dLFs0LDMsIiIsMCx7InN0eWxlIjp7ImJvZHkiOnsibmFtZSI6ImRhc2hlZCJ9fX1dLFs0LDAsImgiLDAseyJzdHlsZSI6eyJib2R5Ijp7Im5hbWUiOiJkYXNoZWQifX19XV0=
			\[\begin{tikzcd}
				Y & {\cU_i} & X \\
				W & {W_i}
				\arrow["{s_i}", from=1-2, to=1-3]
				\arrow["{p_i}", from=1-2, to=2-2]
				\arrow["g", dashed, from=2-1, to=2-2]
				\arrow["f", bend left ,from=1-1, to=1-3]
				\arrow[dashed, from=1-1, to=2-1]
				\arrow["h", dashed, from=1-1, to=1-2]
			\end{tikzcd}\]
		\end{enumerate}
	\end{coro}
	\begin{proof}
		Let $W_i\in\Hilb(X)$ be the closure of the finitely many parametrizing varieties constructed in \cite[Theorem 4.17]{LST} and let $p_i:\cU_i\rightarrow W_i$ be the corresponding universal families.
		Then (1) follows from the same theorem.
		%			\cite[Theorem 4.17]{LST} shows the existence of such families and (1).
		%			In the situation of (2), taking a resolution of indeterminacy, we may assume $Y\dasharrow W$ is a morphism.
		In the situation of (2), let $E$ be the closure of a general fiber of $Y\dasharrow W$ and $E^\prime$ be its image in $X$. Then
		\begin{align*}
			a(X,L)&=a(Y^\prime,L)&\\
			&\geq a(E^\prime,L)&\text{(by Proposition \ref{prop-a-dominant-member})}\\
			&\geq a(E,L)&\text{(by Proposition \ref{prop-a-dominant})}\\
			&=a(Y,f^\ast L)&\text{(by Proposition \ref{prop-adjoint-rigid})}\\
			&=a(X,L),&
		\end{align*}
		which implies that $a(E,f^\ast L)= a(E^\prime,L)$ and thus $E^\prime$ is adjoint rigid as well by Proposition \ref{prop-adjoint-rigid-2}.
		Let $Y^{\prime\prime}$ be the closure of the image of $Y$ in $W\times X$.
		Then the induced morphism $Y^{\prime\prime}\ra W$ is generically flat whose general fiber $E^{\prime\prime}$ is an adjoint rigid subvariety of $X$ with $a(E^{\prime\prime},L)=a(X,L)$.
		%			So $Y^\prime\ra W$ forms a family of adjoint rigid subvarieties of $X$ with $a(E^\prime,f^\ast L)=a(X,L)$ over a dense open subset of $W$ by generic flatness.
		By universal property of Hilbert schemes, $W$ factors rationally through $\Hilb(X)$, whose image lies in $W_i$ for some $i$ by \cite[Theorem 4.17]{LST} and the irreducibility of $W$.
	\end{proof}

	\subsection{Geometric exceptional sets}\label{sec-ges}
	
	The geometric exceptional set was defined for geometrically rational connected varieties in the original papers \cite{LTDuke,LST}.
	For Fano varieties, it reads as follows.
	\begin{defi} {\rm(Geometric exceptional sets for Fano varieties)}\label{defi-ges}
		Let $X$ be a smooth projective Fano variety over a number field $F$ and set $L=-K_X$.
		Then the geometric exceptional set $Z$ is the union of $f(Y(F))$ where $Y$ is a smooth projective variety and $f:Y\rightarrow X$ is a thin map in the following cases:
		\begin{enumerate}[ label={\rm(\arabic*)}]
			\item when $(a(Y,f^\ast L),b(F,Y,f^\ast L))>(a(X,L),b(F,X,L))$; or
			\item when $(a(Y,f^\ast L),b(F,Y,f^\ast L))=(a(X,L),b(F,X,L))$ and either
			\begin{enumerate}[label={\rm(\alph*)}]
				\item $\dim(Y)<\dim(X)$; or
				\item $\dim(Y)=\dim(X)$ and $\kappa(Y,K_Y+a(Y,f^\ast L)f^\ast L)>0$; or
				\item $\dim(Y)=\dim(X)$, $\kappa(Y,K_Y+a(Y,f^\ast L)f^\ast L)=0$ and $f$ is face contracting.
				%			see \cite[Definition 3.5]{LT19a} for the definition of face contracting morphisms).
			\end{enumerate}
		\end{enumerate}
	\end{defi}
	\begin{rema}\label{rema-ges}
		See \cite[Section 5] {LST} for some idea behind Definition \ref{defi-ges}.
		%			If $Y$ is a subvariety of $X$ with $(a(Y,f^\ast L),b(F,Y,f^\ast L))>(a(X,L),b(F,X,L))$ and we do not include $Y(F)$ in $Z$, then the $a$- or $b$-invariant in Manin's conjecture will be violated.
		%			If $Y$ is a subvariety of $X$ with $(a(Y,f^\ast L),b(F,Y,f^\ast L))=(a(X,L),b(F,X,L))$ and we do not include $Y(F)$ in $Z$, then \cite{BL17} shows that the Peyre's constant $c$ will be violated.
		%			However, if we include every thin map $f:Y\rightarrow X$ with $(a(Y,f^\ast L),b(F,Y,f^\ast L))=(a(X,L),b(F,X,L))$, then $Z$ may not be a thin set.
		%			In this case, we only include ones satisfying the face-contracting condition, which can be seen as a refinement of the $b$-invariant.
	\end{rema}

	\section{Geometric exceptional set for Batyrev and Tschinkel's example}\label{sec-the-example}
	
	Let us start by reviewing some notations and introducing new ones.
	We work over a field $F$ of characteristic $0$.
	Let $X$ be the hypersurface in $\bP_x^3\times\bP_y^3$ defined by the equation $$\sum_{i=0}^{3}x_i y_i^3=0$$
	and set $L=-K_X$.
	We denote by $h_1,h_2\in N^1(X)$ the pullbacks, via the projections $\pi_x$ and $\pi_y$, of the hyperplane classes in $\bP^3_x$ and $\bP^3_y$ respectively.
	By Lefschetz hyperplane theorem \cite[Theorem 3.1.17]{Lazarsfeld2004}, $N^1(X)_{\bZ}$ is generated by $h_1$ and $h_2$.
	The variety $X$ is a smooth Fano $5$-fold with $L=-K_X=3h_1+h_2$.
	Thus the $a$- and $b$-invariants of $X$ are $a(X,L)=1$ and $b(F,X,L)=2$.
	We denote by $H_i$ $(i=0,1,2,3)$ the coordinate hyperplanes of $\bP^3_x$.
	%The fibers over $H_i$'s are exactly the singular fibers of $\pi_x$.
	A fiber of $\pi_x$ is smooth if and only if it does not lie above the union of $H_i$.
	
	\subsection{Adjoint rigid subvarieties of $X$}\label{sec-3.1}
	The goal of this section is to prove Proposition \ref{prop-adjoint-rigid-subvarieties}, which classifies subvarieties of $X$ with higher $a$-values and adjoint rigid subvarieties of $X$ with the same $a$-value.
	
	The following lemma studies the $a$-invariant and the adjoint-rigidity of singular fibers of $\pi_x$.
	\begin{lemm}\label{lemm.fiber}
		Let $S$ be the fiber of $\pi_x$ at a closed point $p\in\bP^3_x$. Then
		\begin{enumerate}[ label={\rm(\arabic*)}]
			\item If $p$ is not contained in any $H_i$, then $S$ is a smooth cubic surface and $h^0(S,K_S+L)=1$.
			We have that $a(S,L)=1$ and that $(S,L)$ is adjoint rigid.
			\item If $p$ is contained in exactly one $H_i$, then $S$ is the cone over a smooth cubic curve.%$h^0(S,K_S+L)=1$ I am not sure.
			We have that $a(S,L)=2$ and that $(S,L)$ is not adjoint rigid.
			%		 \item If $x$ is contained in at least two $H_i$, $S$ is a union of planes $H$ and $h^0(H,K_H+L)=0$ over an algebraic closure.
			\item  If $p$ is contained in at least two $H_i$'s, and let $T$ be an irreducible component of $S_\mathrm{red}$.
			Then $a(T,L)=3$ and $(T,L)$ is adjoint rigid.
		\end{enumerate}
	\end{lemm}
	\begin{proof}
		Let $p=(a_1,a_2,a_3,a_4)$.
		Then the cubic surface $S$ is defined by the equation $a_1y_1^3+a_2y_2^3+a_3y_3^3+a_4y_4^3=0$ and
		the number of $H_i$ containing $p$ equals to the number of $i$ such that $a_i=0$.
		
		In case $(1)$, we have that $K_S+L\equiv K_S+h_2\equiv0$. Thus $h^0(S,K_S+L)=1$ and $a(S,L)=1$ by Proposition \ref{prop-a}.
		
		%		In case $(2)$, $S$ is a cone with the unique singular point $(0:0:0:1)$.
		In case $(2)$, the surface $S$ is a cone over a smooth cubic curve $E$.
		Blowing up the cone point $\beta:\widetilde{S}\rightarrow S$ gives a ruled surface $\pi:\widetilde{S}\rightarrow E$.
		%			which must be the canonical map for the pair $(\widetilde{S},L)$.
		I claim that $\pi:\widetilde{S}\rightarrow E$ is birationally equivalent to the rational map $\pi^\prime$ associated to $(\widetilde{S},K_{\widetilde{S}}+a(\widetilde{S},L)\beta^\ast (L))$.
		Indeed, $S$ is not adjoint rigid since it is not rationally connected by Remark \ref{rema-rationally-conn};
		on the other hand, a general fiber of the $\pi^\prime$ is adjoint rigid by Proposition \ref{prop-adjoint-rigid} and so is a rational curve,
		but the only rational curves on $\widetilde{S}$ are the fibers of $\pi$.
		Let $l$ be a general fiber of $\pi$, we have that$(a(S,L)\beta^{\ast} L+K_{\widetilde{S}})\, l=0$ by above arguments, where $K_{\widetilde{S}}\cdot l=-2$ by adjunction and $\beta^\ast L\cdot l=1$.
		We conclude that $a(S,L)=2$.
		%		It follows that $h^0(S,K_S+L)=0$ by Proposition \ref{prop-a}.
		
		In case $(3)$, after base change to the algebraic closure, the surface $\overline{S}$ is a union of planes or a non-reduced plane. 
		%		By \cite[Corollary 4.5]{LST}, the $a$-invariant of a variety is equal to the $a$-invariant of any irreducible component of its base change to the algebraic closure.
		By \cite[Corollary 4.5]{LST}, the $a$-value of $T$ is equal to the $a$-invariant of any irreducible component of $\overline{T}$.
		Thus $a(T,L)=3$, and $T$ is adjoint rigid.
	\end{proof}
	
	Our next goal is to prove Proposition \ref{prop.linearspace}.
	\begin{lemm}\label{lemm-singularity-1}
		Let $Y\subset X$ be the $\pi_x$-preimage of a line $l$ in $\bP_x^3$.
		Assume that $Y\nsubseteq \pi_x^{-1}(H_j) \,(j=0,1,2,3)$, then $Y$ is normal with canonical singularities.
	\end{lemm}
	\begin{proof}
		We may assume that the base field $F$ is algebraically closed.
		%			Since the morphism $\pi_x:X\ra \bP^3_x$ is smooth outside $\bigcup_{j=0}^3 H_j$,
		%			the singular locus of $Y$ lies in the preimage of $4$ closed points $t_j\,(j=1,\cdots,4)$ on $l$, \ie\, the intersection of $\bigcup_{j=0}^3 H_j$ with $l$.
		%			Let $y$ be any singular point of $Y$.
		%			For any $j$, by a change of variables of $\bP^3_x$, we may put $t_j$ to the origin of an affine line $\bA^1_s$.
		The morphism $\pi_x|_{Y}:Y\ra \bP^1$ can be covered by finitely many morphisms of the form of
		%			That is, there exists an affine open $\bA^1_s\subset l$ such that the restriction of $\pi_x$ to $\bA^1_s$ can be identified with
		$Y^\prime\ra \bA^1_s$, where $Y^\prime$ is the hypersurface in $\bA^4_{s,y_1,y_2,y_3}$ defined by
		$$s+\sum_{i=1}^{3}(a_is+b_i)y_i^3=0$$
		for some constants $a_i,b_i\in F $.
		We notice that
		\begin{enumerate}[ label={\rm(\arabic*)}]
			%				\item we may assume that the singular point $y$ lies on the fiber of $\pi_x$ at $s=0$; and
			\item at least one of $a_i$ and $b_i$ is nonzero for each $i$ (by the assumption that $Y\nsubseteq \pi_x^{-1}(H_j)$); and
			\item $b_1,b_2,b_3$ do not vanish simultaneously (since the fiber at $s=0$ is the cubic surface with coefficients $(0:b_1:b_2:b_3)\in\bP^3_x$).
		\end{enumerate}
		
		We first show that $Y$ is normal.
		It is sufficient to show that $\dim \Sing(Y)\leq1$ since $Y$ is a complete intersection subscheme of $X$ \cite[II.Proposition 8.23]{Hartshorne1977AlgebraicG}.
		%			We may denote the $Y^\prime$ associated to $t_j$ constructed above by $Y^\prime_j$.
		%			By construction, we have that $\Sing(Y)\subset S:=\bigcup_{j=0}^3(\Sing(Y^\prime_j)\cap V(s))$.
		This reduces the task to showing that $\dim \Sing(Y^\prime)\leq1$.
		By the Jacobian criterion for smoothness, we obtain that
		$$S:=\Sing (Y^\prime)=\left\{ (0,y_1,y_2,y_3)\mid \sum_{i=1}^{3}a_iy_i^3=-1,\quad b_iy_i=0\,(i=1,2,3) \right\}.$$
		We divide it into three cases.
		\begin{enumerate}[ label={\rm(\arabic*)}]
			\item Suppose that none of $b_i$ is zero, then $S=\emptyset$.
			\item Suppose, without loss of generality, that $b_1=0$ and $b_2,b_3\neq0$. Then $\dim S=0$.
			$$S=\left\{ (0,y_1,0,0)\mid a_1y_1^3=-1 \right\}.$$
			\item Suppose, without loss of generality,  that $b_1=b_2=0$ and $b_3\neq0$. Then $\dim S=1$.
			$$S=\left\{ (0,y_1,y_2,0)\mid a_1y_1^3+a_2y_2^3=-1 \right\}.$$
		\end{enumerate}
		Hence we conclude that $Y$ is normal.
		
		Next, we show that $Y$ has canonical singularities.
		Again, it is enough to show that $Y^\prime$ has canonical singularities.
		By \cite[Theorem 5.34]{Kollr1998}, we only need to show that through any closed point $z\in S$, a general hyperplane section of $Y^\prime$ has canonical singularities.
		We have two cases to consider.
		\begin{enumerate}[ label={\rm(\arabic*)}]
			\item If $b_1=0$ and $b_2,b_3\neq0$, fix a closed point $z\in S$, we have that $z=(0,z_1,0,0)$ and $a_1z_1^3=-1$.
			A general hyperplane $H$ through $z$ is
			$$s+c_1(y_1-z_1)+c_2y_2+c_3y_3=0,$$
			where we may assume that $c_1,c_2,c_3\neq0$.
			The hyperplane section $ Y^\prime \cap H$ is given by the equation
			$$f(y_1,y_2,y_3)=-(c_1(y_1-z_1)+c_2y_2+c_3y_3)(a_1y_1^3+a_2y_2^3+a_3y_3^3+1)+b_2y_2^3+b_3y_3^3,$$
			Now we make a change of coordinates $y_1^\prime:=y_1-z_1$ and consider the polynomial $f_1(y_1^\prime,y_2,y_3):=f(y_1,y_2,y_3)$ in the formal polynomial ring $F [[y_1^\prime,y_2,y_3]]$.
			We have that
			\begin{align*}
				&f_1(y_1^\prime,y_2,y_3)\\
				=&-(c_1y_1^\prime+c_2y_2+c_3y_3)(a_1(y_1^\prime+z_1)^3+a_2y_2^3+a_3y_3^3+1)+b_2y_2^3+b_3y_3^3\\
				=&-(c_1y_1^\prime+c_2y_2+c_3y_3)(\text{(unit)}\cdot y_1^\prime+a_2y_2^3+a_3y_3^3)+b_2y_2^3+b_3y_3^3\\
				=&\text{(unit)}\cdot y_1^{\prime2}-\text{(unit)}\cdot(c_2y_2+c_3y_3)y_1^\prime+b_2y_2^3+b_3y_3^3
			\end{align*}
			in the formal polynomial ring $F [[y_1^\prime,y_2,y_3]]$ by repeatedly using Hensel's lemma to eliminate higher-order terms as in \cite[4.24.(3)]{Kollr1998}.
			By a coordinate change, the scheme defined by $f_1(y_1^\prime,y_2,y_3)$ is isomorphic to the scheme defined by
			$$f_2(y_1^\prime,y_2,y_3):=\text{(unit)}y_1^{\prime2}+\text{(unit)}((c_2y_2+c_3y_3)^2+b_2y_2^3+b_3y_3^3),$$
			where the latter part is of multiplicity $2$.
			Thus by the arguments in \cite[(4.25)]{Kollr1998}, the singularity is canonical.
			
			%		By a further coordinate change, this is isomorphic to the scheme defined by
			%		$$f_3(y_1^\prime,y_2,y_3):=y_1^{\prime2}+y_2^2+y_3^3,$$
			%		which is the $A_2$ singularity.
			
			\item If $b_1=b_2=0$ and $b_3\neq0$, let $z\in  Y^\prime$ be a singular point, then $z=(0,z_1,z_2,0)\in\bA^4$ and $a_1z_1^3+a_2z_2^3=-1$.
			Proceeding analogously, a general hyperplane section of $ Y $ at $z$ is given by
			$$f(y_1,y_2,y_3)=-(c_1(y_1-z_1)+c_2(y_2-z_2)+c_3y_3)(a_1y_1^3+a_2y_2^3+a_3y_3^3+1)+b_3y_3^3.$$
			Now we make a change of coordinates $y_1^\prime:=y_1-z_1$ and $y_2^\prime:=y_2-z_2$, and consider the polynomial $f_1(y_1^\prime,y_2^\prime,y_3):=f(y_1,y_2,y_3)$ in the formal polynomial ring $ F [[y_1^\prime,y_2^\prime,y_3]]$.
			We have that
			\begin{align*}
				&f_1(y_1^\prime,y_2^\prime,y_3)\\
				=&-(c_1y_1^\prime+c_2y_2^\prime+c_3y_3)(a_1(y_1^\prime+z_1)^3+a_2(y_2^\prime+z_2)^3+a_3y_3^3+1)+b_3y_3^3\\
				=&-(c_1y_1^\prime+c_2y_2^\prime+c_3y_3)(\text{(unit)}\cdot z_1^2y_1^\prime+\text{(unit)}\cdot z_2^2y_2^\prime)+b_3y_3^3
				%			=&-(c_1y_1^\prime+c_2y_2^\prime+c_3y_3)(\text{(unit)}\cdot y_1^\prime+\text{(unit)}\cdot y_2^\prime)+b_3y_3^3.
			\end{align*}
			in the formal polynomial ring $F [[y_1^\prime,y_2,y_3]]$.
			It is impossible for both $z_1$ and $z_2$ to be zero, so we may, without loss of generality, assume that $z_1 \neq 0$.
			Regarding $f_1(y_1^\prime,y_2^\prime,y_3)$ as a polynomial in $y_1^\prime$, since the degree $2$ term is nonzero, by a coordinate change, the scheme defined by $f_1$ is isomorphic to the one defined by
			%	$$f_2(y_1^\prime,y_2^\prime,y_3):=-3a_1c_1z_1^2y_1^{\prime2}-((c_2y_2^\prime+c_3y_3)3a_1z_1^2+3a_2c_1z_2^2y_2^\prime)y_1^\prime+(c_2y_2^\prime+c_3y_3)$$
			%$$f_2(y_1^\prime,y_2^\prime,y_3):=
			%\text{(unit)}\cdot y_1^{\prime2}+
			%\text{(unit)}\cdot(y_2^\prime+\text{(unit)}\cdot y_3)^2
			%+\text{(unit)}\cdot(c_2y_2^\prime+c_3y_3)z_2^2y_2^\prime+b_3y_3^3.$$
			$$f_2(y_1^\prime,y_2^\prime,y_3):=\text{(unit)}\cdot y_1^{\prime2}+f_3(y_2^\prime,y_3),$$
			where $f_3$ is a polynomial of multiplicity $2$.
			Thus the singularity is canonical in this case.\qedhere
			%	Proceeding analogously for $y_2^\prime$, the scheme is isomorphic to the one defined by
			%	$$f_3(y_1^\prime,y_2^\prime,y_3):=\text{(unit)}\cdot y_1^{\prime2}+\text{(unit)}\cdot(y_2^{\prime2}+\text{(unit)}\cdot y_3^2),$$
			%	which is the $A_1$ singularity.
		\end{enumerate}
	\end{proof}
	
	\begin{lemm}
		Let $Y\subset X$ be the $\pi_x$-preimage of a plane in $\bP_x^3$.
		Assume that $Y\nsubseteq \pi_x^{-1}(H_j) \,(j=0,1,2,3)$, then $Y$ is normal with canonical singularities.
	\end{lemm}
	\begin{proof}
		Since $\pi_x$ is smooth outside $\bigcup_{j=0}^3 H_j$, the singular locus of $Y$ lies in the preimage of $\bigcup_{j=0}^3 H_j$, which has codimension $1$ in $Y$.
		Thus $Y$ is normal since it is a complete intersection subscheme of $X$ \cite[II.Proposition 8.2.3]{Hartshorne1977AlgebraicG}.
		To show that it is canonical, let $E$ be a divisor over $Y$ at a center $T$ of codimension $\geq2$.
		Let $H\in\left|h_1 \right|$ be a general hyperplane section of $Y$ which intersects $T$.
		Then $K_Y+H$ is Cartier and $H$ is normal with canonical singularities by Lemma \ref{lemm-singularity-1}.
		By inversion of adjunction \cite[Corollary 1.4.5]{BCHM10}, the discrepancy of $(Y,H)$ with respect to $E$ is positive.
		Thus the discrepancy of $Y$ with respect to $E$ is positive as well by \cite[Lemma 2.27]{Kollr1998}.
	\end{proof}
	
	\begin{prop}\label{prop.linearspace}
		Let $Y$ be the $\pi_x$-preimage of a line or plane in $\bP_x^3$ and suppose that $Y$ is not contained in the union of singular fibers of $\pi_x$. Then $a(Y,L)=1$ and $(Y,L)$ is not adjoint rigid.
	\end{prop}
	\begin{proof}
		%			Since $Y$ is normal,
		Since $X$ is Cohen-Macaulay and $Y$ is Cartier,
		we have that $K_Y+L|_Y\sim (3-\dim (\pi_x(Y)))h_1|_Y$ by adjunction.
		Since $Y$ has only canonical singularities by the above two lemmas, we can compute $a(Y,L)$ without taking a resolution of singularities by Proposition \ref{prop-propoties-of-a}.
		Thus the conclusion follows.
	\end{proof}
	
	In the subsequent proof, we will frequently use the following extension theorem of Tsuji and Takayama, which can be seen as a generalization of the Kawamata-Viehweg vanishing.
	
	\begin{theo}{\rm(\cite[Theorem 4.1]{takayama2006pluricanonical})}\label{theo-extension}
		%	Let $X$ be a smooth projective variety and $S$ be a smooth subvariety of $X$.
		%	Assume that $L\sim_\bQ A+E$ is an integral divisor on $X$ with $A$ a big and nef $\bQ$-divisor and $E$ an effective $\bQ$-divisor such that $S$ is of $A$-general position.
		Let $X$ be a smooth projective variety over an algebraically closed field of characteristic $0$ and $S$ be a general member of a generically smooth dominant family of subvarieties of $X$.
		Assume that $L\sim_\bQ A+E$ is an integral divisor on $X$ with $A$ a big and nef $\bQ$-divisor and $E$ an effective $\bQ$-divisor.
		Moreover, assume that $S\not\subset\Supp E$ and $(S,E|_S)$ is a klt pair.
		Then the morphism
		$$H^0(X,m(K_X+S+L))\rightarrow H^0(S,m(K_S+L|_S))$$
		is surjective for any $m>0$.
	\end{theo}
	
	Although the original statement of \cite[Theorem 4.1]{takayama2006pluricanonical} is over $\bC$, it still holds over an arbitrary field of characteristic $0$ by the Lefschetz principle.
	Indeed, the objects in the assumption of Theorem \ref{theo-extension} only involve finitely many elements in the base field, which generate a subfield over $\bQ$ that is isomorphic to a subfield of $\bC$.
	Moreover, the properties appearing in the assumption and conclusion are stable under field extension and descent.
	
	\begin{prop}\label{prop-adjoint-rigid-subvarieties}
		If $Y$ is a closed subvariety of $X$ with $a(Y,L)>1$, then either
		\begin{enumerate}[ label={\rm(\arabic*)}]
			\item $Y$ is contained in the union of singular fibers of $\pi_x$; or
			%		\item $Y$ is a singular fiber of $\pi_x$ in case (2) or (3) of Lemma \ref{lemm.fiber}; or
			%				\item there is a projective variety $B$ and a smooth morphism $\pi:\widetilde{Y}\ra B$ where $\rho:\widetilde{Y}\ra Y$ is a smooth resolution of $Y$, such that $\rho(\pi^{-1}(b))$ is a line in a fiber of $\pi_x$ for every $b\in B$.
			\item $Y$ is the image of a family of subvarieties of $X$ whose general member is a line in a smooth fiber of $\pi_x$ (including the case where $Y$ itself is a line in a smooth fiber of $\pi_x$).
		\end{enumerate}
		If $Y$ is an adjoint rigid closed subvariety with $a(Y,L)=1$, then either
		\begin{enumerate}[resume,label={\rm(\arabic*)}]
			\item $Y=X$; or
			\item $Y$ is contained in $\bigcup_{i=0}^3 H_i$ (\ie\,the union of singular fibers of $\pi_x$); or %(not dominant)
			\item $Y$ is a smooth fiber of $\pi_x$; or %(birational)
			\item $Y$ is a fiber of $\pi_y$; or %(birational)
			\item $Y$ is a conic in a fiber of $\pi_x$. %(line: not dominant, conic: dominant of degree$>1$?)
		\end{enumerate}

	\end{prop}
	%Let $T_0,\cdots,T_3$ denote the preimage of the coordinate hyperplanes in $\bP^3_x$ through $\pi_x$ and $T$ be their union.
	%\begin{prop}
	%	Let $Y$ be a subvariety of $X$ with $a(Y,L)\geq1$.
	%	\begin{enumerate}
		%		\item Suppose $\dim Y=4$.
		%		Then we have $a(Y)=2$ if and only if $Y=T_i$ for $i=0,1,2,3$.
		%%		Suppose further that $a(Y)=1$, then $Y$ is not adjoint rigid.
		%		And there is on subvariety $Y$ of dimension $4$ with $a(Y)=1$ which is adjoint rigid.
		%		\item There is no subvariety $Y$ with $\dim Y=3$, $a(Y,L)\geq1$ and $Y\cap T=\emptyset$.
		%		\item We have $\dim Y=2$ and $Y\cap T=\emptyset$ if and only if $Y$ is in the following cases.
		%		\begin{enumerate}[label=\alph*)]
			%			\item $Y$ is a smooth fiber of $\pi_y$, in which case we have $a(Y)=b(Y)=1$.
			%			\item $Y$ is a smooth fiber of $\pi_x$. The $b$-invariant of such $Y$ is recorded in Lemma \ref{lemm.fiber}.
			%		\end{enumerate}
		%		\item  We have $\dim Y=1$ and $Y\cap T=\emptyset$ if and only if $Y$ is in the following cases.
		%		\begin{enumerate}[label=\alph*)]
			%			\item $Y$ is a line on a smooth fiber of $\pi_x$, in which case we have $a(Y)=2$.
			%			\item $Y$ is a conic on a smooth fiber of $\pi_x$, in which case we have $a(Y)=1$.
			%		\end{enumerate}
		%	\end{enumerate}
	%\end{prop}
	
	\begin{proof}
		%	Let $Y$ be the resolution of singularities of a subvariety $Y^\prime$ of $X$.
		Let $Y$ be a subvariety of $X$.
		We divide the proof into several cases based on the dimension of $Y$.
		Since the projection maps are flat and $X$ dominates each $\bP^3$, by dimension counting we have that $\dim\pi_i(Y)\geq \dim Y-2$.
		We will use this formula to make a further classification in each case.
		
		%	\textbf{Singular fibers of $\pi_x$.}
		%	Suppose $Y$ is contained in the union of singular fibers of $\pi_x$ with $a(Y,L)>1$, then the analysis in Lemma \ref{lemm.fiber} shows that $Y$ is either in case (2) 
		%	In the subsequent proof, we can assume that $Y$ is not contained in the union of singular fibers of $\pi_x$. Indeed, by the analysis in Lemma \ref{lemm.fiber}, we know that such $Y$ are covered by $h_2$-lines, which have $a$-invariant equals to $2$ and is contained in (6) of the proposition.
		We do not need a full classification of $Y$ in the cases (1) and (4) of Proposition \ref{prop-adjoint-rigid-subvarieties} through out this paper, so
		we may assume that $Y$ is not contained in the union of singular fibers of $\pi_x$ throughout the proof.
		
		\textbf{Dimension 5.}
		We have that $X$ is adjoint rigid with $a(X,L)=1$.
		
		\textbf{Dimension 4.}
		When $\dim Y=4$, we have that $2\leq\dim\pi_i(Y)\leq 3$ for $i=x$ or $y$.
		
		\textbf{Case 4.A.}
		Suppose that $Y$ dominates neither of $\bP^3$'s, then $Y\cong S_1\times S_2$ for some surfaces $S_1$ and $S_2$. This is impossible since then $\pi_y(\pi_x^{-1}(p))$ will be constant with respect to $p\in S_1$, but by the definition of $X$, there are no two points in $\bP^3_x$ correspond to the same surface in $\bP^3_y$.
		
		\textbf{Case 4.B.}
		Now suppose that $Y$ dominates $\bP^3_x$.
		
		%	\textbf{(\rmnum{1}) Show ??.}
		%			We assume that $Y$ is not contained in the union of singular fibers of $\pi_x$ throughout this case.
		Let $\rho:\widetilde{Y}\rightarrow Y$ be a resolution of singularities of $Y$ and abuse notations of $\pi_i$ and $h_i$.
		Then $\widetilde{Y}$ is smooth and irreducible.
		Let $Y_1$ be the intersection of $\widetilde{Y}$ with a general member of $\left| h_1 \right|$.
		We may assume that $Y_1$ is smooth and irreducible by Bertini's theorem.
		
		%%????extension theorem??	
		%	Since $2h_1+h_2$ is big and nef on $\widetilde{Y}$, we may use Kawamata-Viehweg vanishing for the short exact sequence
		%% https://q.uiver.app/#q=WzAsNSxbMSwwLCJcXGNPKEtfe1lee1xccHJpbWV9fSsyTC1oXzEpIl0sWzIsMCwiXFxjTyhLX3tZXlxccHJpbWV9KzJMKSJdLFszLDAsIlxcY08oS197WV8xfSsyTC1oXzEpIl0sWzQsMCwiMCJdLFswLDAsIjAiXSxbNCwwXSxbMCwxXSxbMSwyXSxbMiwzXV0=
		%\[\begin{tikzcd}
			%	0 & {\cO(K_{\widetilde{Y}}+2h_1+h_2)} & {\cO(K_{\widetilde{Y}}+3h_1+h_2)} & {\cO(K_{Y_1}+2h_1+h_2)} & 0
			%	\arrow[from=1-1, to=1-2]
			%	\arrow[from=1-2, to=1-3]
			%	\arrow[from=1-3, to=1-4]
			%	\arrow[from=1-4, to=1-5]
			%\end{tikzcd}\]
			%	to obtain the inequality
			%	$$h^0(\widetilde{Y},K_{\widetilde{Y}}+3h_1+h_2)\geq h^0(Y_1,K_{Y_1}+2h_1+h_2).$$
			%	Proceeding analogously, let $Y_2$ be the intersection of $Y_1$ with a general member of $\left| h_1 \right|$.
			%	We have
			%	$$h^0(Y_1,K_{Y_1}+2h_1+h_2)\geq h^0(Y_2,K_{Y_2}+h_1+h_2).$$
			
			%	By the same token in the proof of Theorem \ref{theo-a-cover}, it follows by the extension theorem (Theorem \ref{theo-extension}) that
			
			We may write $2h_1+h_2\sim_\bQ A+E$, where $A\sim_\bQ 2h_1+\frac{1}{2}h_2$ is big and nef and $E\sim_\bQ \frac{1}{2}h_2$ is effective.
			We may assume that $2E|_{Y_1}$ is a smooth and integral divisor of $Y_1$ by Bertini's theorem.
			Then $(Y_1,E|_{Y_1})$ is a klt pair by \cite[Corollary 2.31]{Kollr1998},
			and Theorem \ref{theo-extension} implies that
			$$h^0(Y,m(K_{\widetilde{Y}}+3h_1+h_2))\geq h^0(Y_1,m(K_{Y_1}+2h_1+h_2))$$
			for any $m>0$.
			Let $Y_2$ be the intersection of $Y_1$ with a general member of $\left| h_1 \right|$.
			By the same token, we have that
			$$h^0(Y,m(K_{\widetilde{Y}}+3h_1+h_2))\geq h^0(Y_2,m(K_{Y_2}+h_1+h_2))$$
			for any $m>0$.
			This shows that $a(Y,L)\leq a(Y_2,h_1+h_2)$ and $\kappa(\widetilde{Y},K_{\widetilde{Y}}+L)\geq \kappa(Y_2,K_{Y_2}+h_1+h_2)$.
			
			%				Now $\dim\pi_x(Y_2)=1$ implies that a general fiber of $\pi_x|_{Y_2}$ is a dimension $1$ subscheme of fibers of $\pi_x:X\rightarrow \bP^3_x$.
			%	Let $Y_2\overset{g}{\rightarrow} W\rightarrow\bP^1$ be the Stein factorization of $\pi_x:Y_2\rightarrow \bP^1$.
			Suppose that $a(Y,L)>1$. We have that $a(Y_2,h_1+h_2)>1$.
			Let $\nu:Y_2^\nu\rightarrow Y_2^\prime$ be the normalization map of the image $Y_2^\prime$ of $Y_2$ in $X$, then $h_1+h_2$ is ample on $Y_2^\nu$ since a normalization map is finite.
			So the polarized variety $(Y_2^\nu,h_1+h_2)$ is isomorphic to one of the four cases listed in Proposition \ref{prop-large-a} (in dimension $n=2$).
			Since there is a surjective morphism $\pi_x:Y_2^\nu\rightarrow \bP^1$ (where we abuse the notation of $\pi_x$), we can not have $Y_2^\nu\cong \bP^2$.
			Thus case (1) and (3) in Proposition \ref{prop-large-a} are impossible.
			So there exists a rational map from $(Y_2^\nu,h_1+h_2)$ to a smooth curve $C$ whose general fiber $f$ has intersection $1$ with $(h_1+h_2)$.
			Thus either $f\cdot h_1=1$ and $f\cdot h_2=0$, or $f\cdot h_1=0$ and $f\cdot h_2=1$. 
			In other words, the image of $f$ in $X$ is an $h_1$-line in a fiber of $\pi_y$, or an $h_2$-line in a fiber of $\pi_x$.
			Suppose that it is the former.
			Then, we must have $\dim(\pi_y(Y_2))=1$ since $\dim(Y_2)=2$.
			Hence $Y_2$ is isomorphic to $C_1\times C_2$ for some line $C_1\subset \bP_x^3$ and curve $C_2\subset \bP_y^3$.
			By the generality in the definition of $Y_2$, this implies that $Y$ is isomorphic to $\bP_x^3\times C_2$.
			But there is no such $Y$ inside $X$, as it is easy to find three smooth fibers of $\pi_x$ whose intersection is empty.
			So we conclude that the image of $f$ in $X$ is an $h_2$-line in a fiber of $\pi_x$.
			Let $\alpha: \widetilde{Y}\ra B\ra \pi_x(Y)$ be the Stein factorization of $\pi_x|_{Y}\circ\rho$ and let $B_2$ be the intersection of $B$ with the pullback of the same two $h_1$ as in the definition of $Y_2$.
			Then we have the following diagram:
			% https://q.uiver.app/#q=WzAsNSxbMSwxLCJCXzIiXSxbMiwxLCJDXzEiXSxbMSwwLCJZXzIiXSxbMiwwLCJZXzJeXFxwcmltZSJdLFswLDAsIkMiXSxbMiwzLCJcXHJob18yIl0sWzMsMSwiXFxwaV8xfF97WV8yXlxccHJpbWV9Il0sWzIsMCwiXFxhbHBoYV8yIiwyXSxbMCwxXSxbMiw0LCJcXHBpIiwyLHsic3R5bGUiOnsiYm9keSI6eyJuYW1lIjoiZGFzaGVkIn19fV1d
			\[\begin{tikzcd}
				C & {Y_2} & {Y_2^\prime} \\
				& {B_2} & {C_1}
				\arrow["\pi"', dashed, from=1-2, to=1-1]
				\arrow["{\rho_2}", from=1-2, to=1-3]
				\arrow["{\alpha_2}"', from=1-2, to=2-2]
				\arrow["{\pi_x|_{Y_2^\prime}}", from=1-3, to=2-3]
				\arrow[from=2-2, to=2-3]
			\end{tikzcd}\]
			%			where a general fiber of $\pi_x|_{Y_2^\prime}$ and a general fiber of $\pi_x|_{Y_2^\prime}\circ \rho_2$ are birational.
			Since $\alpha_2$ contracts a general fiber of $\pi$, there exists a rational map $\varphi:C\dasharrow B_2$ by rigidity lemma \cite[Lemma 1.15]{Debarre2001}.
			Since a general fiber of both $\pi$ and $\alpha_2$ is smooth and connected, the map $\varphi$ is birational.
			By generality of the two $h_2$ in the definition of $Y_2$, we conclude that $\rho(\alpha^{-1}(b))$ is a line in a smooth $\pi_x$-fiber for a general closed point $b\in B$.

			We have shown that if $Y$ is in neither case (1) or (2) of Proposition \ref{prop-adjoint-rigid-subvarieties}, then $a(Y_2,h_1+h_2)\leq1$.
			Now suppose that $a(Y_2,h_1+h_2)=1$ and that $(Y_2,h_1+h_2)$ is adjoint rigid.
			Let $\nu:Y_2^\nu\rightarrow Y_2^\prime$ be the normalization map of the image $Y_2^\prime$ of $Y_2$ in $X$, then $h_1+h_2$ is ample on $Y_2^\nu$.
			It follows by \cite[Lemma 5.3]{Lehmann2019} that $Y_2^\nu$ has canonical singularities and $-K_{Y_2^\nu}\equiv h_1+h_2$.
			Thus $Y_2$ is a weak del Pezzo surface, and we have that $-K_{Y_2}\equiv h_1+h_2$ since the resolution is crepant.
			Moreover, since $Y$  dominates $\bP^3_x$, we have that $Y\equiv ah_1+bh_2$ for some $a\geq0$ and $b\geq1$.
			Suppose that $Y\equiv h_2$.
			Then a general fiber of $(\pi_x|_{Y}\circ \rho)|_{Y_2}:Y_2\ra \bP^1$ is a smooth curve of genus $1$.
			So $Y_2$ is an elliptic surface, which can not be a weak del Pezzo surface since $K_{Y_2}^2\leq0$.
			Thus we have that $a,b\geq1$ and so
			$$K_{Y_2}^2= (h_1+h_2)^2 h_1^2 (ah_1+bh_2)(h_1+3h_2) = 3a+7b \geq 10,$$
			but the volume of the anticanonical divisor of a weak del Pezzo surface must be $\leq9$.
			%				On the other hand, the adjunction formula implies that  $K_{Y_2}^\nu\equiv K_X+Y+2h_1\equiv (a-1)h_1+(b-1)h_2$.
			%				These force that $Y\equiv 0$, a contradiction.
			We conclude that if $a(Y_2,h_1+h_2)=1$, then $\kappa(Y,K_Y+L)\geq\kappa(Y_2,K_{Y_2}+h_1+h_2)\geq1$, \ie, $(Y,L)$ is not adjoint rigid.

			\textbf{Case 4.C.}
			The remaining case is when $Y$ maps to a surface in $\bP^3_x$ and  dominates $\bP^3_y$, which is equivalent to $Y\equiv ch_1$ for some positive integer $c$.
			In this case, $Y_2$ may be reducible and we denote its irreducible components by $S_j$ $(j=1\cdots,m)$.
			Note that $S_j$ equals $\pi_x^{-1}(p_j)$ for some closed point $x_j\in\bP_x^3$.
			As we assumed that $Y$ is not contained in $\bigcup_{i=0}^3 H_i$, we may further assume that 
			$S_j$ is a smooth fiber of $\pi_x$ for each $j$ by the generality of the two hyperplane sections in the definition of $Y_2$.
			By the Kawamata-Viehweg vanishing, we have that
			\begin{equation}
				\begin{split}\label{eq1}
					h^0(\widetilde{Y},K_{\widetilde{Y}}+3h_1+h_2)\geq &h^0(Y_2,K_{Y_2}+h_1+h_2)\\
					=&\sum_{j=1}^{m}h^0(S_j,K_{S_j}+h_2).
				\end{split}
			\end{equation}
			Suppose that the right-hand side of (\ref{eq1}) is zero. Then it follows by Lemma \ref{lemm.fiber} that each $x_j$ is contained in some $H_i$,
			which contradicts our assumption on $S_j$.
			%				This force that $\pi_x(Y)\subseteq\bigcup_{i=0}^3 H_i$ by the generality, \ie, $Y$ is contained in the union of singular fibers of $\pi_x$.
			
			%	Take the stein factorization $Y_2\rightarrow Z\rightarrow \bP^3$ of $\pi_y$, then each fiber of $Y_2\rightarrow Z$ maps to a plane in $\bP^3_y$ by $\pi_y$.
			%	Suppose the right hand side is zero, then the intersection of two general $h_1$ would be contained in two or more $H_i$'s, which is impossible by dimension reasons.
			
			Thus if $Y$ is not contained in the union of singular fibers of $\pi_x$, then
			$$h^0(\widetilde{Y},K_{\widetilde{Y}}+3h_1+h_2)\geq 1,$$
			which implies that $a(Y,L)\leq 1$.
			Moreover, since $Y_2$ is defined by general hyperplane sections, we may assume that each $S_j$ is a smooth $\pi_x$-fiber.
			Then, the right-hand side of (\ref{eq1}) equals $c$.
			When $h^0(\widetilde{Y},K_{\widetilde{Y}}+3h_1+h_2)\geq2$, $Y$ is not adjoint rigid by definition.
			When $h^0(\widetilde{Y},K_{\widetilde{Y}}+3h_1+h_2)=1$ we must have $c=1$ by (\ref{eq1}).
			%				By Proposition \ref{prop.linearspace}, $Y$ is not adjoint rigid unless possible $Y$ is contained in the union of singular fibers of $\pi_x$.
			In this case, Proposition \ref{prop.linearspace} and the assumption that $Y$ is not contained in the union of singular fibers of $\pi_x$ implies that $Y$ is not adjoint rigid.
			
			\textbf{Dimension $3$}.
			Suppose that $\pi_x(Y)$ has dimension $3$ or $2$.
			The argument in this case is similar to that in Case 4.
			Let $Y_1$ be the intersection of $\widetilde{Y}$ with a general member of $\left|h_1\right|$, and let $Y_2$ be the intersection of $\widetilde{Y}$ with two general member of $\left|h_1\right|$.
			As in Case 4.B, Theorem \ref{theo-extension} implies that
			$$h^0(\widetilde{Y},m(K_{\widetilde{Y}}+3h_1+h_2))\geq h^0(Y_2,m(K_{Y_2}+h_1+h_2))$$
			for any $m\geq1$.
			Suppose that $\pi_x(Y)$ has dimension $3$.
			The situation is similar to that in Case 4.B, but we can use a simpler argument.
			We may assume that $Y_2$ is a smooth and irreducible curve by Bertini's theorem.
			By Lefschetz hyperplane theorem \cite[Theorem 3.1.17]{Lazarsfeld2004},
			$Y$ is numerical equivalent to $ah_1^2+bh_1h_2+ch_2^2$ for some non-negative integers $a,b,$ and $c$.
			Since we assumed that $\dim(\pi_x(Y))=3$, we have that $c\geq1$.
			Then a direct computation of intersection numbers shows that $\deg(h_1+h_2)=3b+4c\geq4$ on $Y_2$.
			This and the Riemann-Roch theorem imply that
			$$h^0(Y_2,K_{Y_2}+h_1+h_2)=h^0(Y_2,-h_1-h_2)+\deg(h_1+h_2)-1+g(Y_2)\geq 3.$$
			Therefore, we have that $a(Y,L)\leq1$, and when $a(Y,L)=1$, the variety $Y$ is not adjoint rigid.
			
			Suppose that $\pi_x(Y)$ has dimension $2$.
			There are several ways to conclude that there is no such $Y$ satisfying the assumptions of Proposition \ref{prop-adjoint-rigid-subvarieties} by mimicking the above arguments.
			A quick way is to notice that $Y\equiv ah_1^2+bh_1h_2$ for some integers $a\geq0$ and $b\geq1$.
			Since $2h_1+h_2$ is ample on $Y_1^\nu$, as above, \cite[Lemma 5.3]{Lehmann2019} implies that either
			\begin{itemize}
				\item $a(Y_1^\nu,2h_1+h_2)<1$; or
				\item the pair $(Y_1^\nu,2h_1+h_2)$ is not adjoin rigid; or
				\item $Y_1$ is a weak del Pezzo surface and $-K_{Y_1}\equiv 2h_1+h_2$.
			\end{itemize}
			Suppose that either $a(Y,L)>1$, or $(Y,L)$ is adjoint rigid with $a(Y,L)=1$, then the only possibility is the third case.
			But a computation of intersection numbers shows that
			$$K_{Y_1}^2=(2h_1+h_2)^2(ah_1^2+bh_1h_2)h_1(h_1+3h_2)=3a+13b\geq13,$$
			which contradicts the fact that $Y_1$ is a weak del Pezzo surface.

			%				Similar to Case 4.C, write $C_j$ $(j=1,\cdots,m)$ for the irreducible components of $Y_2$, and we may assume that $C_j$ lies on a smooth fiber of $\pi_x$.
			%				
			%				Let $\widetilde{C_j}$ be the normalization of $C_j$.
			%				Then
			%				\begin{equation}\label{eq-h0}
				%					h^0(Y_2,K_{Y_2}+h_1+h_2)=\sum_{j=1}^{b}h^0(\widetilde{C_j},K_{\widetilde{C_j}}+h_2),
				%				\end{equation}

			When $\pi_x$ maps $Y$ to a curve and $\pi_y$ maps $Y$ to a surface, we must have $Y^\prime\cong C\times S$ for some curve $C$ and surface $S$, but there is no such subvariety of $X$.
			
			So the only remaining case is when $\pi_x$ maps $Y$ to a curve and $\pi_y$ is dominant,  which means that $Y=ch_1^2$ for some positive integer $c$.
			The argument is almost the same as in Case.4.C, except that we take a general hyperplane section of $Y$ only once.
			We conclude that $a(Y,L)\leq1$ if $Y$ is not contained in the union of singular fibers of $\pi_x$,
			and that when $a(Y,L)=1$, the pair $(Y,L)$ is not adjoint rigid if $Y\not\equiv h_1^2$.
			When $Y\equiv h_1^2$, we have that $Y$ is also not adjoint rigid by Proposition \ref{prop.linearspace}.
			
			\textbf{Dimension $2$}.
			We have three cases. 
			If $Y$ is a fiber of $\pi_y$, then $Y\cong \bP^2$. We have that $(Y,L)$ is adjoint rigid with $a(Y,L)=1$.
			If both $\pi_x$ and $\pi_y$ maps $Y$ to a curve, then the volume of $L$ is
			$$\vol(L)=L^2\cdot Y=(9h_1^2+6h_1h_2+h_2^2)Y>9.$$
			By \cite[Lemma 5.3]{Lehmann2019}, either $(Y,L)$ is not adjoint rigid or $a(Y,L)<1$.
			The remaining case is when $Y$ is a fiber of $\pi_x$.
			But this case is covered by cases (4) and (5) of Proposition \ref{prop-adjoint-rigid-subvarieties}.
			See Lemma \ref{lemm.fiber} for a further analysis of this cases.
			%				For a further analysis of the $a$-value and the adjoint-rigidness of such $Y$ in Lemma \ref{lemm.fiber}.
			
			\textbf{Dimension $1$}.
			When $Y$ is a curve, the assumption $a(Y,L)\geq1$ implies that $Y$ is rational. So $a(Y,L)=2/(Y\cdot L)\geq1$ if and only if $Y\cdot h_1=0$ and $Y\cdot h_2=2$ or $1$.
			In other words, $Y$ is a conic or line in a fiber of $\pi_x$.
			When $Y$ is a line, the pair $(Y,L)$ is adjoint rigid with $a(Y,L)=2$.
			When $Y$ is a conic, $(Y,L)$ is adjoint rigid with $a(Y,L)=1$.
		\end{proof}
		
		%\begin{prop}
		%	Let $D$ be the deformation of lines on fibers of $\pi_x$.
		%	Let $\overline{X}$ and $\overline{D}$ be the base change to the algebraic closure $\overline{\bQ}$.
		%	Then we have $\piet(\overline{X}\backslash \overline{D})=1$.
		%\end{prop}
		%\begin{proof}
		%	For simplicity we will omit the bar.
		%	We denote the restriction of $\pi_x$ to $X\backslash D$ still by $\pi_x:X\backslash D\rightarrow \bP^3_x$.
		%	Then $\pi_x:X\backslash D\rightarrow \bP^3_x$ is flat and proper since the two properties are stable under base change.
		%	Let $s:\Spec\,\overline{\bQ}\rightarrow \bP^3$ be a geometric point over a general rational point of $\bP^3$ and $x$ be a geometric point on $(X\backslash D)_{s}$.
		%	By the homotopy exact sequence
		%	$$\piet((X\backslash D)_{s},x)\rightarrow\piet(X\backslash D,x)\rightarrow \piet(\bP^3,s)\rightarrow 1,$$
		%	to show $\piet(X\backslash D,x)=1$, it is enough to show $\piet((X\backslash D)_{s},x)=1$.
		%	Now $(X\backslash D)_{s}$ is a smooth cubic surface $S$ removing a line $l$.
		%	The linear system of planes passing through $l$ defines a conic fibration $S\backslash l \rightarrow \bP^1$, which is proper and flat.
		%	So we have the homotopy exact sequence
		%		$$\piet((S\backslash l)_y,x)\rightarrow\piet(S\backslash l,x)\rightarrow \piet(\bP^1,y)\rightarrow 1,$$
		%		where $y$ denotes the corresponding geometric point on $\bP^1$.
		%		Note that $\piet((S\backslash l)_y,x)=1$ since $(S\backslash l)_y$ is a smooth conic.
		%		Thus we are enough to conclude that  $\piet(X\backslash D,x)=1$.
		%\end{proof}% This is false, since the conic fibration morphism is not projective.
		
		\begin{rema}
			Although we do not have a complete classification when $Y$ is contained in the union of singular fibers of $\pi_x$, this does not affect the computation of the geometric exceptional set $Z$ since singular fibers of $\pi_x$ are covered by $h_2$-lines, which are adjoint rigid with $a$-invariant equal to $1$, and so all rational points on them are included in $Z$. 
		\end{rema}

		\subsection{Adjoint rigid $a$-covers of $X$}\label{sec-3.2}
		
		\begin{defi}
			Let $f:Y\rightarrow X$ be a dominant generically finite morphism of smooth projective varieties of degree $\geq2$ and $L$ be a big and nef $\bQ$-divisor on $X$.
			We say that $f:Y\rightarrow X$ is an adjoint rigid $a$-cover of $(X,L)$ if $(Y,f^\ast L)$ is adjoint rigid and $a(Y,f^\ast L) = a(X,L)$.
		\end{defi}
		
		The goal of this subsection is to show the following theorem.
		\begin{theo}\label{theo-a-cover}
			With the notations in Notation \ref{notation}, there does not exist adjoint rigid $a$-cover of $(X,L)$.
		\end{theo}
		
		%We need the following result on the branch divisor of an a-cover.
		%\begin{theo}
		%	\label{theo-sen}
		%	{\rm(\cite[Corollary 2.8]{sengupta2021manin})}
		%	Let $X$ be a smooth projective geometrically uniruled variety with a big and nef divisor $L$ and $f:Y\rightarrow X$ be an adjoint rigid a-cover.
		%	Then the branch divisor $B\subset X$ of $f:Y\rightarrow X$ satisfies $a(B,L)>a(X,L)$.
		%\end{theo}
		In \cite[Theorem 5.5]{beheshti2022moduli}, the authors show that there is no adjoint rigid $a$-cover of Fano threefolds with anticanonical polarization.
		We will use the extension theorem (Theorem \ref{theo-extension}) to reduce our case to that of Fano threefolds.

		\begin{proof}[Proof of Theorem \ref{theo-a-cover}]
			We may assume that the base field is algebraic closed by Proposition \ref{prop-propoties-of-a}.
			%	Let 
			%	% https://q.uiver.app/#q=WzAsMyxbMCwwLCJZIl0sWzEsMCwiVyJdLFsyLDAsIlxcYlBeMyJdLFswLDEsIlxccGkiXSxbMSwyLCJcXHJobyJdXQ==
			%	\[\begin{tikzcd}
				%		Y & W & {\bP^3}
				%		\arrow["\pi", from=1-1, to=1-2]
				%		\arrow["\rho", from=1-2, to=1-3]
				%	\end{tikzcd}\]
			%	 be the Stein factorization of $\pi_x\circ f:Y\rightarrow \bP^3$.
			%				Suppose that $f:Y\rightarrow X$ is an adjoint rigid $a$-cover of $(X,3h_1+h_2)$.
			Let $f:Y\rightarrow X$ be an adjoint rigid $a$-cover of $(X,3h_1+h_2)$.
			%				We may assume $f$ is surjective by Remark \ref{rema-surjective}.
			Let $Y_1$ be a general member of $f^\ast \left|h_1 \right|$, then $f$ induces a surjective generally finite morphism $f_1:Y_1\rightarrow X_1$ where we can assume that $X_1$ is a general member of $\left|h_1\right|$. By Bertini's theorem, we can assume that $Y_1$ and $X_1$ are smooth and irreducible.
			
			We may write $f^\ast(2h_1+h_2)\sim_\bQ A+E$, where $A:=f^\ast(2h_1+\frac{1}{2}h_2)$ is big and nef and $E:=\frac{1}{2}f^\ast(h_2)$ is effective.
			%	A general member of $f^\ast| h_1|$ is not contained in the proper closed subset $Y^c_A$ of $Y$ (see \cite[Section 2.4]{takayama2006pluricanonical} for the definition of $Y^c_A$).
			It is easy to check the remaining assumptions in Theorem \ref{theo-extension} are satisfied and thus we obtain
			$$h^0(Y_1,m(K_{Y_1}+2h_1+h_2))\leq h^0(Y,m(K_Y+3h_1+h_2))$$
			for any $m>0$. This shows that $\kappa(Y_1,K_{Y_1}+2h_1+h_2)\leq\kappa(Y,K_Y+3h_1+h_2)=0$.
			By the same token, we can show that
			$$\kappa(Y_2,K_{Y_2}+h_1+h_2)\leq\kappa(Y,K_Y+3h_1+h_2)=0,$$
			where $Y_2$ is a complete intersection of two general members of $f^\ast\left| h_1 \right|$.
			By Riemann-Hurwitz formula, we have that $K_{Y_2}+h_1+h_2\equiv K_{Y_2}-f^\ast K_{X_2} \equiv R$ for some effective divisor $R$,
			which forces that $\kappa(Y_2,K_{Y_2}+h_1+h_2)=0$.
			This shows that $a(Y_2,h_1+h_2)=1$ since the divisor $K_{Y_2}+h_1+h_2$ lies on the pseudo-effective boundary.
			Moreover, $X_2$ is a Fano threefold since $-K_{X_2}\equiv h_1+h_2$ is ample.
			Thus $f_2:Y_2\rightarrow X_2$ is an adjoint rigid $a$-cover of a Fano threefold with anticanonical polarization, which does not exist by \cite[Theorem 5.5]{beheshti2022moduli}.
		\end{proof}

		\subsection{$b$-invariant of subvarieties of $X$}\label{sec-3.3}
		The study of the Picard rank of smooth diagonal cubic surfaces can be traced back to Segre (see also \cite[Exercise 3.12]{debarre2017geometry}).
		\begin{lemm}
			{\rm (\cite[Section III]{Seg43})}\label{lemm-segre}
			Let $S$ be the smooth cubic surface in $\bP^3$ defined over a field $F$ of characteristic $\neq 3$ by the equation
			$$a_1y_1^3+a_2y_2^3+a_3y_3^3+a_4y_4^3=0,$$
			with $a_1a_2a_3a_4\neq0$.
			Then the Picard rank of $S$ is $1$ if and only if
			$$\frac{a_{\tau(1)}a_{\tau(2)}}{a_{\tau(3)}a_{\tau(4)}}$$
			is not a cube in $F$ for any permutation $\tau\in S_4$.
		\end{lemm}
		\begin{rema}\label{rema-segre}
			One can characterize $S$ with any Picard rank by methods in \cite{PT00}.
			The Picard rank of smooth cubic surfaces is always $7$ over an algebraically closed field of characteristic $0$.
			Over $\bQ$, the Picard rank of $S$ can take integers between $1$ and $4$;
			and over a number field, it can also be $5$ or $7$.
		\end{rema}

		We need the following non-existence result of adjoint rigid $a$-cover to prove Theorem \ref{theo-factor-through}.
		\begin{lemm}\label{lemm-a-cover-of-subvariety}
			Let $(X,L)$ be defined as in Notation \ref{notation}.
			Let $Y$ be an adjoint rigid closed subvariety of $X$ satisfying $a(Y,L)=a(X,L)$ and $b(F,Y,L)<b(F,X,L)$.
			Suppose that $Y$ is not contained in the union of singular fibers of $\pi_x$.
			Then $Y$ is either
			\begin{enumerate}[ label={\rm(\arabic*)}]
				\item conics in a fiber of $\pi_x$, or
				\item fibers of $\pi_y$, or
				\item a fiber of $\pi_x$ which is a cubic surface of Picard rank $1$.
			\end{enumerate}
			Moreover, there does not exist adjoint rigid $a$-cover of such $Y$.
		\end{lemm}
		\begin{proof}
			For the first part, all possibilities of adjoint rigid subvarieties $Y$ with $a(Y,L)=1$ is listed in (3)-(7) of Proposition \ref{prop-adjoint-rigid-subvarieties}.
			For reader's convenience, we recall them as follows:
			\begin{enumerate}[label={\rm(\arabic*)}]
				\setcounter{enumi}{2}
				\item $Y=X$; or
				\item $Y$ is contained in the union of singular fibers of $\pi_x$; or %(not dominant)
				\item $Y$ is a smooth fiber of $\pi_x$; or %(birational)
				\item $Y$ is a fiber of $\pi_y$; or %(birational)
				\item $Y$ is a conic in a fiber of $\pi_x$. %(line: not dominant, conic: dominant of degree$>1$?)
			\end{enumerate}
			In the above cases, case (3) is impossible due to the assumption that $b(F,Y,L)<b(F,X,L)$,
			and case (4) is impossible by assumption.
			Let $Y$ be as in case (5). Then $b(F,Y,L)=\rho(X)$ and $b(F,X,L)=2$, which imply that $Y$ falls under case (3) of Lemma \ref{lemm-a-cover-of-subvariety}.
			Let $Y$ be as in case (6), then $b(F,Y,L)=b(F,\bP^2,\cO(3))=\rho(\bP^2)=1$, , which corresponds to case (2) of Lemma \ref{lemm-a-cover-of-subvariety}.
			Let $Y$ be as in case (7), then $b(F,Y,L)=b(F,\bP^1,\cO(2))=\rho(\bP^1)=1$, which corresponds to case (1) of Lemma \ref{lemm-a-cover-of-subvariety}.

			For the second part, (1) and (2) are projective spaces, and they do not admit any $a$-cover by \cite[Example 7.1]{LT19}.
			(3) is a del Pezzo surface with anticanonical polarization, which do not admit any adjoint rigid $a$-cover by \cite[Theorem 6.2]{LTDuke}.
		\end{proof}
		
		%??Lemma??
		%\begin{lemm}\label{lemm-subvariety-of-chow}
		%	Let $(X,L)$ be defined as in Example \ref{exam-BT}.
		%	Let $W^\circ$ be a subvariety of $\Chow(X)$ parametrizing subvarieties $Y\subset X$ satisfying $a(Y,L)=a(X,L)$ and $(Y,L)$ is adjoint rigid.
		%	Let $W$ be the closure of $W^\circ$ in $\Chow(X)$.
		%	Then each subvariety $[Y]\in W$ satisfies $a(Y,L)\geq a(X,L)$.
		%\end{lemm}
		%\begin{proof}
		%	This is a direct corollary of our classification in \ref{prop-adjoint-rigid-subvarieties}.
		%\end{proof}
		
		\subsection{Geometric monodromy action}\label{sec-monodromy}
		Let $F$ be a field of characteristic $0$ and $\overline{F}$ be its algebraic closure.
		Fix a primitive cube root of unit $\omega$.
		Let $W\ra \bP_x^3$ be a generically finite morphism.
		Let $\overline{F}(\overline{W})$ denote the function field of $\overline{W}$.
		Define $\pi_x^\prime:X^\prime:=X\times_{\bP_x^3}W\ra W$ by base change.
		Then the function field of $\overline{\bP_x^3}$ is $$\overline{F}(\overline{\bP_x^3})\cong\overline{F}\left(\frac{x_1}{x_0},\frac{x_2}{x_0},\frac{x_3}{x_0}\right)=\overline{F}(\alpha_1^3,\alpha_2^3,\alpha_3^3).$$
		where we fix three cubic roots
		$$\alpha_1=\sqrt[\uproot{10}3]{\frac{x_1}{x_0}},\quad
		\alpha_2=\sqrt[\uproot{10}3]{\frac{x_2}{x_0}},\quad
		\alpha_3=\sqrt[\uproot{10}3]{\frac{x_3}{x_0}}.$$
		Let $T:=\bP_x^3\ra \bP_x^3$ be the finite morphism defined by $x_i\mapsto x_i^3$.
		The corresponding Galois group is
		$$\Gal(\overline{F}(\alpha_1,\alpha_2,\alpha_3)/\overline{F}(\alpha_1^3,\alpha_2^3,\alpha_3^3))=\langle \tau_1,\tau_2,\tau_3 \rangle\cong\bZ/3\bZ\times\bZ/3\bZ\times\bZ/3\bZ,$$
		where $\tau_i:\alpha_i\mapsto \omega\alpha_i$ for each $i$.
		Denote
		$$\gamma_1=\sqrt[\uproot{10}3]{\frac{x_0x_3}{x_1x_2}}=\frac{\alpha_3}{\alpha_1\alpha_2},\quad
		\gamma_2=\sqrt[\uproot{10}3]{\frac{x_0x_1}{x_2x_3}}=\frac{\alpha_1}{\alpha_2\alpha_3},\quad
		\gamma_3=\sqrt[\uproot{10}3]{\frac{x_0x_2}{x_1x_3}}=\frac{\alpha_2}{\alpha_1\alpha_3}.$$

		Let $W^\prime$ be the image of $W\ra\bP_x^3$
		and let $W\ra W^{\prime\prime}\ra W^\prime$ be the Stein factorization.
		%			Let $W^\prime$ be the image of $W\ra\bP_x^3$
		%			and let $I\subset \overline{F}[\alpha_1^3,\alpha_2^3,\alpha_3^3]$ be the prime ideal defining $\overline{W^\prime}$.
		Then the morphism $W\ra W^{\prime\prime}$ is birational and the morphism $W^{\prime\prime}\ra W^\prime$ is finite and dominant.
		These induce maps of affine coordinate rings
		$$\overline{F}[\alpha_1^3,\alpha_2^3,\alpha_3^3]
		\ra \overline{F}[\overline{W^\prime}]
		\hookrightarrow \overline{F}[\overline{W^{\prime\prime}}].$$
		%			Then we have a map of integral rings $$\overline{F}[\alpha_1^3,\alpha_2^3,\alpha_3^3]\twoheadrightarrow \overline{F}[\alpha_1^3,\alpha_2^3,\alpha_3^3]/I\hookrightarrow \overline{F}[\overline{W}]\hookrightarrow \overline{F}(\overline{W}),$$
		%			where $\overline{F}[\overline{W}]$ denotes the coordinate ring of $\overline{W}$.
		We fixed the cubic roots $\gamma_i$ so that we can choose an irreducible component $\overline{T^\prime}$ of $\overline{W}\times_{\overline{\bP_x^3}}\overline{T}$ canonically.
		This induce the following maps of affine coordinate rings:
		% https://q.uiver.app/#q=WzAsNSxbMCwxLCJcXG92ZXJsaW5le0Z9W1xcYWxwaGFfMV4zLFxcYWxwaGFfMl4zLFxcYWxwaGFfM14zXSJdLFsxLDEsIlxcb3ZlcmxpbmV7Rn1bXFxvdmVybGluZXtXXlxccHJpbWV9XSJdLFsyLDEsIiBcXG92ZXJsaW5le0Z9W1xcb3ZlcmxpbmV7V157XFxwcmltZVxccHJpbWV9fV0iXSxbMCwwLCJcXG92ZXJsaW5le0Z9W1xcYWxwaGFfMSxcXGFscGhhXzIsXFxhbHBoYV8zXSJdLFsyLDAsIlxcb3ZlcmxpbmV7Rn1bXFxvdmVybGluZXtUXlxccHJpbWV9XSJdLFswLDFdLFsxLDIsIiIsMCx7InN0eWxlIjp7InRhaWwiOnsibmFtZSI6Imhvb2siLCJzaWRlIjoidG9wIn19fV0sWzAsMywiIiwyLHsic3R5bGUiOnsidGFpbCI6eyJuYW1lIjoiaG9vayIsInNpZGUiOiJ0b3AifX19XSxbMyw0XSxbMiw0LCIiLDAseyJzdHlsZSI6eyJ0YWlsIjp7Im5hbWUiOiJob29rIiwic2lkZSI6InRvcCJ9fX1dXQ==
		\[\begin{tikzcd}
			{\overline{F}[\alpha_1,\alpha_2,\alpha_3]} && {\overline{F}[\overline{T^\prime}]} \\
			{\overline{F}[\alpha_1^3,\alpha_2^3,\alpha_3^3]} & {\overline{F}[\overline{W^\prime}]} & { \overline{F}[\overline{W^{\prime\prime}}]}
			\arrow[from=1-1, to=1-3]
			\arrow[hook, from=2-1, to=1-1]
			\arrow[from=2-1, to=2-2]
			\arrow[hook, from=2-2, to=2-3]
			\arrow[hook, from=2-3, to=1-3]
		\end{tikzcd}\]
		The vertical map on the right-hand side induces a field extension $\overline{F}(\overline{W})\hookrightarrow \overline{F}(\overline{T^\prime})$.
		By abuse of notation, we denote the image of $x
		_i,\alpha_i$ and $\gamma_i$ in $\overline{F}(\overline{T^\prime})$ still by the same notations.
		
		Let $\overline{X^\prime}_{\eta}:=\overline{X}\times_{\Spec\,\overline{F}}\Spec\,\kappa(\eta)$ be the generic fiber of $\pi_x^\prime$ which is the smooth cubic surface defined by the equation
		$$x_0y_0^3+x_1y_1^3+x_2y_2^3+x_3y_3^3=0$$
		over the field $\overline{F}(\overline{W})$.
		Then $\overline{F}(\overline{T^\prime})$ is a splitting field of $\overline{X^\prime}_{\eta}$.
		By Lemma \ref{lemm-segre}, $\overline{X^\prime}_{\eta}$ has Picard rank $\geq2$ if and only if $\gamma_i\in\overline{F}(\overline{W})$ for some $i\in\{1,2,3\}$.
		
		Write $\overline{X^\prime}_{\overline{\eta}}:=\overline{X}\times_{\Spec\,\overline{F}}\Spec\,\overline{\kappa(\eta)}$ for the geometric generic fiber of $\pi_x^\prime$, which is a smooth cubic surface over the algebraically closure of $\kappa(\eta)=\overline{F}(\overline{W})$.
		Let $W^\circ$ be a Zariski open subset of $W$ over which $\pi_x^\prime:X^\prime\ra W$ is smooth.
		For any geometric point $\overline{w}\in \overline{W^\circ}$, there exists a specialization map
		\begin{eqnarray}\label{eq:sp}
			N^1(\overline{X^\prime}_{\overline{\eta}})\ra N^1({\overline{X^\prime}}_{\overline{w}})
		\end{eqnarray}
		between the N\'eron-Severi spaces by \cite[Proposition 3.6]{MP12}.
		
		\begin{lemm}\label{lemm-monodromy-galois}
			Under the above notations,
			%				Let $X^\prime_w$ be a smooth fiber of $\pi$ over a closed point $w\in W^\circ$ and fix a geometric point $\overline{w}$ over $w$.
			%				Let $X^\prime_\eta$ be the generic fiber of $\pi_x^\prime$.
			there exists a geometric point $\overline{w}\in \overline{W^\circ}$ such that the specialization map (\ref{eq:sp}) induces an isomorphism
			$$N^1(\overline{X^\prime}_{\eta})\ra N^1({\overline{X^\prime}}_{\overline{w}})^{\piet(\overline{W^\circ},\overline{w})}.$$
		\end{lemm}
		\begin{proof}
			%				The base change $\overline{X^\prime}_{\overline{\eta}}\ra\overline{X^\prime}_{\eta}$ induce a map
			%				\begin{equation}\label{eq-sp}
				%%					N^1(\overline{X^{\prime}})\ra
				%					N^1(\overline{X^\prime}_{\eta})
				%					\ra N^1({\overline{X^\prime}}_{\overline{\eta}})
				%				\end{equation}
			%				between N\'eron-Severi spaces.
			%				The specialization map (\ref{eq:sp}) coincides with the restriction map (\ref{eq-sp}) by the constructions in the proof of \cite[Proposition 3.3]{MP12}.
			%				Since $\pi_x^\prime$ is smooth over $W^\circ$, the image of (\ref{eq-sp})
			%				equals $N^1({\overline{X^\prime}}_{\overline{w}})^{\piet(\overline{W^\circ},\overline{w})}$ by Deligne's invariant cycle theorem \cite[V. Th\'eor\`eme 16.24]{voisin2002theorie}.
			%				(Although the setting of \cite[V. Th\'eor\`eme 16.24]{voisin2002theorie} is over $\bC$, it holds over $\overline{F}$ by the Lefschetz principle.)
			The Hochschild-Serre spectral sequence induces a long exact sequence
			% https://q.uiver.app/#q=WzAsNSxbMiwwLCJcXFBpYyh7XFxvdmVybGluZXtYXlxccHJpbWV9fV97XFxvdmVybGluZXtcXGV0YX19KV5HIl0sWzEsMCwiXFxQaWMoe1xcb3ZlcmxpbmV7WF5cXHByaW1lfX1fe1xcZXRhfSkiXSxbMywwLCJcXEJyKFxca2FwcGEoXFxldGEpKSJdLFswLDAsIjAiXSxbNCwwLCJcXGNkb3RzIl0sWzEsMF0sWzAsMl0sWzIsNF0sWzMsMV1d
			\[\begin{tikzcd}
				0 & {\Pic({\overline{X^\prime}}_{\eta})} & {\Pic({\overline{X^\prime}}_{\overline{\eta}})^G} & {\Br(\kappa(\eta))} & \cdots
				\arrow[from=1-1, to=1-2]
				\arrow[from=1-2, to=1-3]
				\arrow[from=1-3, to=1-4]
				\arrow[from=1-4, to=1-5]
			\end{tikzcd}\]
			where $G:=\Gal(\overline{\kappa(\eta)}/\kappa(\eta))$ is the Galois group.
			The N\'eron-Severi spaces of ${\overline{X^\prime}}_{\eta}$ and ${\overline{X^\prime}}_{\overline{\eta}}$ are nothing but the Picard groups tensored with $\bR$, since the varieties involved are rationally connected.
			After tensoring with $\bR$, the Brauer group vanishes, and we obtain an isomorphism
			$$N^1({\overline{X^\prime}}_{\eta})\cong N^1({\overline{X^\prime}}_{\overline{\eta}})^G.$$
			
			Moreover, there exists a canonical surjection $G=\piet(\eta,\overline{\eta})\twoheadrightarrow \piet(\overline{W^\circ},\overline{\eta})$
			by \cite[\href{https://stacks.math.columbia.edu/tag/0BQM}{Proposition 0BQM}]{stacks-project}.
			Thus, the Galois action on $N^1({\overline{X^\prime}}_{\overline{\eta}})$ factors through the monodromy action,
			and there is an equality
			$$N^1({\overline{X^\prime}}_{\overline{\eta}})^G=N^1({\overline{X^\prime}}_{\overline{\eta}})^{\piet(\overline{W^\circ},\overline{\eta})}$$
			of their invariant parts.
			
			By \cite[Theorem 1.1 and Remark 1.3]{MP12},
			there exists a geometric point $\overline{w}\in \overline{W^\circ}$ such that the specialization map $(\ref{eq:sp})$ is an isomorphism.
			Furthermore, the map $(\ref{eq:sp})$ is compatible with the monodromy action by the construction in the proof of \cite[Proposition 3.3]{MP12}.
			We thus obtain an isomorphism
			$$N^1({\overline{X^\prime}}_{\overline{\eta}})^{\piet(\overline{W^\circ},\overline{\eta})}\cong
			N^1({\overline{X^\prime}}_{\overline{w}})^{\piet(\overline{W^\circ},\overline{w})}$$
			between the monodromy invariant parts.
			The proof is completed by combining the above results.
		\end{proof}

		We will need the following descent property of birational morphisms in the proof of Theorem \ref{theo-factor-through}.
		\begin{prop}\label{prop-fpqc-descent}
			Let $f:X\ra Y$ be a morphism of geometrically integral schemes over a field $F$.
			Let $f^\prime:X^\prime\ra Y^\prime$ be the base change of $f$ along a field extension $F\hookrightarrow F^\prime$.
			%		Write $f^\prime:X^\prime:=X\times_S S^\prime\ra Y^\prime:=Y\times_S S^\prime$ for the morphism obtained by base change.
			If $f^\prime$ is birational, then so is $f$.
		\end{prop}
		\begin{proof}
			We fix notations as follows:
			% https://q.uiver.app/#q=WzAsNCxbMCwwLCJYXlxccHJpbWUiXSxbMSwwLCJZXlxccHJpbWUiXSxbMCwxLCJYIl0sWzEsMSwiWSJdLFswLDEsImZeXFxwcmltZSJdLFsyLDMsImYiXSxbMCwyLCJoIiwyXSxbMSwzLCJnIl1d
			\begin{equation*}
				\begin{tikzcd}
					{X^\prime} & {Y^\prime} \\
					X & Y
					\arrow["{f^\prime}", from=1-1, to=1-2]
					\arrow["h"', from=1-1, to=2-1]
					\arrow["g", from=1-2, to=2-2]
					\arrow["f", from=2-1, to=2-2]
				\end{tikzcd}
			\end{equation*}
			Since $f^\prime$ is a birational morphism between integral schemes, there exist nonempty open subsets $U^\prime\subset X^\prime$ and $V^\prime\subset Y^\prime$ such that $f^\prime$ induces an isomorphism between $U^\prime$ and $V^\prime$.
			Since any field extension is faithfully flat, the same holds for $g$ and $h$.
			Thus the topological images $V:=g(V^\prime)$ and $U:=h(U^\prime)$ are nonempty open and we take the induced open subscheme structures.
			We have that $$f(U)=f(h(U^\prime))=g(f^\prime(U^\prime))=g(V^\prime)=V.$$
			Since the fiber product of schemes is glued by affine open covers from the target (see \cite[II. Theorem 3.3]{Hartshorne1977AlgebraicG}),
			the scheme $U^\prime$ is identified with the fiber product of $f|_U$ and $g|_{V^\prime}$.
			%				The diagram (\ref{diag-fpqc}) induces a Cartesian square
			%				% https://q.uiver.app/#q=WzAsNCxbMSwxLCJWIl0sWzEsMCwiVl5cXHByaW1lIl0sWzAsMSwiZl57LTF9KFYpIl0sWzAsMCwiVV57XFxwcmltZX0iXSxbMiwwXSxbMywxXSxbMSwwXSxbMywyXV0=
			%				\[\begin{tikzcd}
				%					{U^{\prime}} & {V^\prime} \\
				%					{f^{-1}(V)} & V
				%					\arrow[from=1-1, to=1-2]
				%					\arrow[from=1-1, to=2-1]
				%					\arrow[from=1-2, to=2-2]
				%					\arrow[from=2-1, to=2-2]
				%				\end{tikzcd}\]
			%				of topological spaces, and we take the induced open subscheme structure.
			By the fpqc descent of the property ``isomorphism" \cite[Proposition 14.51]{ulrich}, the morphism $f:|_U:U\ra V$ is an isomorphism.
			Thus $f$ is birational since $X$ and $Y$ are integral.
		\end{proof}

		We isolate the most technical part of the proof of Theorem \ref{theo-factor-through} into the following theorem.
		\begin{theo}\label{theo-monodromy-final}
			We follow Notation \ref{notation} and Notation \ref{nota-thin-set}.
			Let $f:Y\ra X$ be a generically finite morphism of degree $\geq2$ from a smooth geometrically integral variety $Y$ such that
			\begin{enumerate}[label={\rm(\arabic*)}]
				\item $(Y,f^\ast L)$ is not adjoint rigid and satisfies $a(Y,f^\ast L)=a(X,L)=1$ and $b(F,Y,f^\ast L)\geq b(F,X,L)=2$. \label{mono-assu-1}
				\item the rational map $\rho:X\dasharrow W$ associated to $(Y,K_Y+a(Y,f^\ast L)f^\ast L)$ factors rationally through $\pi_x:X\ra \bP_x^3$, and a general fiber of $\rho$ maps dominantly to a smooth fiber of $\pi_x$. \label{mono-assu-2}
			\end{enumerate}
			Then $f:Y\ra X$ factors rationally through $f_\tau:X_\tau\ra X$ for some $\tau\in\mathfrak{S}_4$.
		\end{theo}
		
		\begin{proof}
			\textit{Step 1: reduce the problem to show that $W^\prime\ra \bP_x^3$ factors rationally through $T_i\ra \bP_x^3$.}
			First of all, only three morphisms $T_i\ra \bP_x^3$ $(i=1,2,3)$ are involved among the morphisms $T_\tau\ra \bP_x^3$ $(\tau\in\mathfrak{S}_4)$, which are defined at the beginning of Section \ref{sec-proof-of-thm}.
			Let $W^\prime\ra \bP_x^3$ be a resolution of indeterminacy of the rational map $W\dasharrow \bP_x^3$ asserted in (2) of Theorem \ref{theo-monodromy-final} such that $W^\prime$ is smooth.
			Let $\Gamma\ra W^\prime$ be the graph of $Y\dasharrow W^\prime$.
			Write $X^\prime:=X\times_{\bP_x^3} W$.
			The scheme $X^\prime$ is irreducible since $W\ra \bP_x^3$ is an algebraic fiber space, see Remark \ref{rema-algebraic-fiber-space}.
			We thus obtain a morphism $\Gamma\ra X^\prime$ by the universal property of fiber products.
			To prove that $f:Y\ra X$ factors rationally through $f_i:X_i\ra X$, it is sufficient to prove that there exists a rational map $X^\prime\ra X_i$ which commutes with $X^\prime\ra X$ and $X_i\ra X$.
			This, in turn, reduces the problem to proving the existence of a rational map $W^\prime\dasharrow T_i$ that commutes with $W^\prime\ra\bP_x^3$ and $T_i\ra\bP_x^3$.
			For reader's convenience, we summarize the known morphisms into the following commutative diagram.
			In the diagram, ``bir." denotes a birational morphism, and a Cartesian square is indicated by a square.
			
			\[ \begin{tikzcd}
				\Gamma\rar["\textrm{ bir.}"]\dar["\gamma"] &Y\dar["f"]&\\
				X^\prime\rar["f^\prime"]\dar\drar[phantom, "\square"]  & X\dar["\pi_x"]\drar[phantom, "\square"] & X_i\dar\lar["f_i"swap] \\%
				W^\prime \rar["g^\prime"]& \bP_x^3& T_i\lar["g_i"swap]
			\end{tikzcd}
			\]
			
			\textit{Step 2: show that $\gamma:\Gamma\ra X^\prime$ is birational.}
			%				From now on, we denote $\widetilde{W}$ simply as $W$ for convenience.
			It is sufficient to prove that the base change $\overline{\Gamma}\ra \overline{X^\prime}$ to an algebraic closure $\overline{F}$ is birational by Proposition \ref{prop-fpqc-descent}.
			Let $\overline{w}\in\overline{W^\prime}$ be a general closed point and $\overline{a}$ be its image in $\overline{\bP_x^3}$.
			Then the restriction $\overline{\Gamma}_{\overline{w}}\rightarrow \overline{X}_{\overline{a}}$ is a surjection of adjoint rigid varieties with the same $a$-value by Proposition \ref{prop-adjoint-rigid} and \ref{prop-adjoint-rigid-2}.
			So $\overline{\Gamma}_{\overline{w}}\rightarrow \overline{X}_{\overline{a}}$ must be birational since there does not exist adjoint rigid $a$-cover of an anticanonically polarized smooth cubic surface by \cite[Theorem 6.2]{LTDuke}.
			Thus the morphism $\overline{\Gamma}_{\overline{w}}\ra\overline{X^\prime}_{\overline{w}}$ is birational since
			the morphism $\overline{X^\prime}_{\overline{w}}\rightarrow \overline{X}_{\overline{a}}$ is isomorphic.
			
			We now prove that $\overline{\gamma}:\overline{\Gamma}\ra\overline{X^\prime}$ is dominant and generically finite.
			We have two dominant morphisms $\overline{\Gamma}\ra \overline{W^\prime}$ and $\overline{X^\prime}\ra \overline{W^\prime}$ such that the restriction $\overline{\gamma}_{\overline{w}}$ of $\gamma$ to a general closed point $\overline{w}\in\overline{W^\prime}$ is a birational morphism.
			Thus the morphism $\overline{\gamma}$ must be dominant.
			Moreover, the morphism $\overline{\gamma}$ is generically finite since the composition $\overline{\Gamma}\ra \overline{Y}\ra \overline{X}$ is generically finite.

			We now prove that $\overline{\gamma}:\overline{\Gamma}\ra\overline{X^\prime}$ is birational.
			Since $\overline{\gamma}$ is generically finite and dominant, by generic flatness \cite[\href{https://stacks.math.columbia.edu/tag/0529}{Tag 0529}]{stacks-project}, there exists a dense open subset $U$ of $\overline{X^\prime}$ such that for any closed point $x\in U$, the preimage $\overline{\gamma}^{-1}(x)$ is a finite set with exactly $d=\deg(\overline{\gamma})$ points.
			Since $\overline{\Gamma_{\overline{w}}}\ra \overline{X^\prime}_{\overline{w}}$ is birational for general closed point $\overline{w}\in\overline{W^\prime}$,
			%				There must be a closed point $\overline{w}\in\overline{W^\prime}$ such that $\overline{\Gamma_{\overline{w}}}\ra \overline{X}_{\overline{w}}$ is birational, and that $\overline{X^\prime}_{\overline{w}}\cap U$ is dense in $\overline{X^\prime}_{\overline{w}}$.
			we must have $d=1$.
			Thus $\overline{\gamma}:\overline{\Gamma}\ra\overline{X^\prime}$ is birational.

			\textit{Step 3: show that $g^\prime$ factors rationally through $g_i$.}
			The three morphisms $T_i\ra \bP^3$ correspond to the field extensions ${F}(\alpha_1^3,\alpha_2^3,\alpha_3^3)\hookrightarrow F(\alpha_1^3,\alpha_2^3,\alpha_3^3,\gamma_i)$ for $i\in\{1,2,3\}$
			Write $\bP_x^{3\circ}:=\bP_x^3\backslash V(x_0x_1x_2x_3))$ and let $W^{\prime\circ}$ be the pullback of $\bP_x^{3\circ}$ to $W^\prime$.
			Since $L$ is ample on $X$, its pullback to $\Gamma$ is big and semiample.
			Thus $a(\Gamma_w,L|_{\Gamma_w})$ is constant for $w\in W^{\prime\circ}$ by \cite[Theorem 3.14]{LT19}.
			By Proposition \ref{prop-adjoint-rigid}, we have that $a(X^\prime_w,f^{\prime\ast}L)=a(X^\prime,f^{\prime\ast}L)$ for any closed point $w\in W^\circ$.
			This is assumption (3) of Proposition \ref{prop-b-inv-and-monodromy}. The other three assumptions are satisfied as well:
			\begin{enumerate}[label={\rm(\arabic*)}]
				\item $W^\prime$ is smooth by construction.
				\item $X^\prime\ra W^\prime$ is smooth over $W^{\prime\circ}$ since it is base changed from a smooth morphism.
				\setcounter{enumi}{3}
				\item $X^\prime\ra W^\prime$ is birationally equivalent to $Y\ra W$ by construction.
			\end{enumerate}
			The conclusion of Step 2 and the assumption (1) of Theorem \ref{theo-monodromy-final} imply that
			$$b(F,X^\prime,f^{\prime\ast}L)=b(F,Y,f^\ast L)\geq2.$$
			This, together with Proposition \ref{prop-b-inv-and-monodromy}
			and Lemma \ref{lemm-monodromy-galois}, shows that there exists a geometric point $\overline{w}\in\overline{W^{\prime\circ}}$ such that
			\begin{align*}
				\dim (N^1(\overline{X^\prime}_{\overline{\eta}}))=\dim(N^1(\overline{X^\prime}_{\overline{w}})^{\piet(\overline{W^{\prime\circ}},\overline{w})})&
				\geq b(F,X^\prime,f^{\prime\ast}L)\geq 2,
				%					&\text{(by Lemma \ref{lemm-monodromy-galois})}\\
				%					&=\dim( \overline{F}(\overline{W})[\mathscr{E}]^{\Gal(\overline{F}(\overline{T^\prime})/\overline{F}(\overline{W}))})-2.&\text{(by Lemma \ref{lemm-etale-algebra-2})}
			\end{align*}
			which implies that $\gamma_i\in \overline{F}(\overline{W^\prime})$ for some $i\in\{1,2,3\}$ by Lemma \ref{lemm-segre} (see also the arguments at the beginning of Section \ref{sec-monodromy}).
			Since we fixed a cubic root $\gamma_i$, we can choose an irreducible component $T^\prime_i$ of $W^\prime\times_{\bP_x^3}T_i$ canonically:
			%				Define $T^\prime$ by the following Cartesian diagram.
			% https://q.uiver.app/#q=WzAsNCxbMSwwLCJ7VF9pfSJdLFsxLDEsIntcXGJQX3heM30iXSxbMCwwLCJ7VF5cXHByaW1lX2l9Il0sWzAsMSwie1deXFxwcmltZX0iXSxbMCwxLCJnX2kiLDAseyJzdHlsZSI6eyJoZWFkIjp7Im5hbWUiOiJlcGkifX19XSxbMiwzLCIiLDAseyJzdHlsZSI6eyJoZWFkIjp7Im5hbWUiOiJlcGkifX19XSxbMiwwXSxbMywxLCJnXlxccHJpbWUiXV0=
			\[\begin{tikzcd}
				{{T^\prime_i}} & {{T_i}} \\
				{{W^\prime}} & {{\bP_x^3}}
				\arrow[from=1-1, to=1-2]
				\arrow[two heads, from=1-1, to=2-1]
				\arrow["{g_i}", two heads, from=1-2, to=2-2]
				\arrow["{g^\prime}", from=2-1, to=2-2]
			\end{tikzcd}\]
			Then $\gamma_i\in \overline{F}(\overline{W^\prime})$ by the arguments at the beginning of Section \ref{sec-monodromy}.
			Thus $\overline{T^\prime_i}\ra\overline{W^\prime}$ is birational, and so is
			$T_i^\prime\ra W^\prime$ by Proposition \ref{prop-fpqc-descent}.
			Hence there is a rational section of $T_i^\prime\ra W^\prime$.
			Thus $W^\prime$ factors rationally through $T_i$ over $\bP_x^3$ as desired.
		\end{proof}

		\subsection{Proof of Theorem \ref{theo-factor-through} and Corollary \ref{coro-thin-set}}\label{sec-proof-of-thm}
		\begin{proof}[Proof of Theorem \ref{theo-factor-through}]
			Define closed subvarieties $T_1,T_2$ and $T_3$ of $\bP^3_x\times \bP^1_{s,t}$ by
			\begin{align*}
				T_1:=&\left\{s^3x_{0}x_{1}-t^3x_{2}x_{3}=0\right\},\\
				T_2:=&\left\{s^3x_{0}x_{2}-t^3x_{1}x_{3}=0\right\},\\
				T_3:=&\left\{s^3x_{0}x_{3}-t^3x_{1}x_{2}=0\right\}.
			\end{align*}
			It is enough to consider $T_1,T_2$ and $T_3$ instead of $T_\tau$ for each $\tau\in\mathfrak{S}_4$.
			Let $f_i:X_i:=X\times_{\bP_x^3}T_i\rightarrow X$ be the base change along $\pi_x:X\rightarrow \bP_x^3$.
			%				Define $$Z:=V(F)\cup \bigcup_{i=1}^3 f_i(X_i(F)),$$
			%				where $V$ is the union of subvarieties $Y\subset X$ with $a(Y,L)>a(X,L)$.
			By the classification results in Proposition \ref{prop-adjoint-rigid-subvarieties}, we see subvarieties $Y\subset X$ with $a(Y,L)>a(X,L)$ are contained in
			\begin{align*}
				%					V_0:=&\left\{x_0x_1x_2x_3=0\right\},\\
				V_1:=&\left\{x_{0}y_{0}^3+x_{1}y_{1}^3=x_{2}y_{2}^3+x_{3}y_{3}^3=0\right\},\\
				V_2:=&\left\{x_{0}y_{0}^3+x_{2}y_{2}^3=x_{1}y_{1}^3+x_{3}y_{3}^3=0\right\},\\
				V_3:=&\left\{x_{0}y_{0}^3+x_{3}y_{3}^3=x_{1}y_{1}^3+x_{2}y_{2}^3=0\right\},
			\end{align*}
			where $V_1,V_2$ and $V_3$ are the closure of families of lines in smooth fibers of $\pi_x$.
			(Note that the union of singular fibers of $\pi_x$ is already contained in each of $V_i$ $(i=1,2,3)$.)
			In fact, the union of subvarieties $Y\subset X$ with $a(Y,L)>a(X,L)$ is exactly the set $\bigcup_{i=1}^3 V_i$, since the latter set is covered by $h_2$-lines in $\pi_x$-fibers, whose $a$-invariant equals $2$.
			
			Let $Y$ be a smooth projective variety and $f:Y\rightarrow X$ be a thin map with $(a(Y,L),b(Y,L))\geq(a(X,L),b(X,L))$ and denote the image of $Y$ in $X$ by $Y^\prime$.
			We will show that $f$ factors through $f_i$ or $V_i$ for some $i$.
			By Proposition \ref{prop-a-dominant}, we have that $a(Y^\prime,L)\geq a(Y,L)\geq a(X,L)$.
			If $a(Y^\prime,L)>a(X,L)$, then $f$ factors through $V_i$ for some $i$.
			Thus we can suppose throughout the proof that $a(Y^\prime,L)=a(Y,L)=a(X,L)$.
			
			When $(Y,L)$ is adjoint rigid, we have that $Y^\prime$ is adjoint rigid as well by Proposition \ref{prop-adjoint-rigid-2}.
			If further $b(Y^\prime,L)< b(X,L)$, there is no adjoint rigid $a$-cover of $Y^\prime$ by Lemma \ref{lemm-a-cover-of-subvariety}, which implies that $\deg f=1$ and thus $b(Y,L)=b(Y^\prime,L)< b(X,L)$, a contradiction.
			
			When $(Y,L)$ is adjoint rigid and $b(Y^\prime,L)\geq b(X,L)$, we need to consider the case of $Y^\prime=X$.
			Since there is no adjoint rigid $a$-cover of $X$ by Theorem \ref{theo-a-cover}, the morphism $f$ in this case is dominant and birational, which is not a thin map.
			Thus we can assume that $Y^\prime$ is a proper closed subvariety of $X$.
			Since $Y^\prime$ is adjoint rigid with $a(Y^\prime,L)=a(X,L)$ and $b(F,Y^\prime,L)\geq b(F,X,L)$, by Proposition \ref{prop-adjoint-rigid-subvarieties}, either
			\begin{itemize}
				%					\item $Y^\prime$ is contained in the union of singular fibers of $\pi_x$, which is $V_0$; or
				\item $Y^\prime$ is contained in the union of singular fibers of $\pi_x$, which is	contained in $V_i$ for each of $i$; or
				\item $Y^\prime$ is a line in a $\pi_x$-fiber, which is contained in $\bigcup_{i=1}^3 V_i$; or
				\item $Y^\prime$ is a smooth $\pi_x$-fiber with  $b(F,Y^\prime,L)\geq2$, which is contained in $\bigcup_{i=1}^3 f_i(X_i(F))$ by Lemma \ref{lemm-segre}.
			\end{itemize}
			In each case, the morphism $f$ factors through $V_i$ or $f_i$ for some $i$.
			This finishes the case when $(Y,L)$ is adjoint rigid.
			
			It remains to show the case when $(Y,L)$ is not adjoint rigid with $a(Y,L)=a(Y^\prime,L)=a(X,L)$.
			(Note that in this case, $(Y^\prime,L)$ can be adjoint rigid or not.)
			%	Take a smooth resolution $\widetilde{Y}\rightarrow Y$ and denote its composition with $f$ by $\widetilde{f}$.
			%	By Lang-Nishimura theorem, it is sufficient to show $\widetilde{f}(\widetilde{Y}(F))\subset Z$ since $\widetilde{Y}\rightarrow Y$ is birational.
			%				\textbf{\color{red} [This is false, since singular points on $Y$ may not be lift. Singular rational points which can be lift are called singularities of type R.
				%				Every non-cyclic tame quotient singularity in dimension $2$ are of type R.
				%				But there exists cyclic singularity in dimension $2$ and characteristic $0$ which is no of type R. So it is necessary to assume $Y$ is smooth.]}
			Let $\rho:{Y}\dashrightarrow W$ be the rational map associated to $({Y},K_{{Y}}+a(Y,L)L)$.
			%				Let $\Gamma$ be the graph of $\rho$ with canonical projections $s:\Gamma\rightarrow {Y}$ and $p:\Gamma\rightarrow W$.
			%				Then $(p:\Gamma\rightarrow W, {f} \circ s:\Gamma\rightarrow X)$ is a family generically parametrizing adjoint rigid subvarieties of $X$ with $a$-value equals to $X$.
			
			%				Now $\Gamma\rightarrow W$ is a family of projective varieties such that $\Gamma_w$ is adjoint rigid and $a(\Gamma_w,L)=a(Y,L)=a(Y^\prime,L)$ for general $w\in W$, and the morphism $Y\rightarrow Y^\prime$ is dominant.
			There exists a family $(s_i:\cU_i\rightarrow X,p_i:\cU_i\rightarrow W_i)$ asserted in Corollary \ref{coro-family-breaking} such that $f$ factors rationally through $s_i$.
			Since we only need to show the rationally factoring property,
			we may take a resolution of indeterminacy.
			We thus obtain the following commutative diagram where $s$ is birational.

			% https://q.uiver.app/#q=WzAsNixbMCwwLCJcXEdhbW1hIl0sWzEsMSwiWSJdLFswLDIsIlciXSxbMiwwLCJcXGNVX2kiXSxbMiwyLCJXX2kiXSxbMywxLCJYIl0sWzAsMiwicCIsMl0sWzAsMSwicyJdLFszLDUsInNfaSJdLFszLDQsInBfaSIsMix7ImxhYmVsX3Bvc2l0aW9uIjoyMH1dLFswLDNdLFsxLDUsImYiLDAseyJsYWJlbF9wb3NpdGlvbiI6MjB9XSxbMiw0LCIiLDAseyJzdHlsZSI6eyJib2R5Ijp7Im5hbWUiOiJkYXNoZWQifX19XSxbMSwyLCIiLDIseyJzdHlsZSI6eyJib2R5Ijp7Im5hbWUiOiJkYXNoZWQifX19XV0=
			\[\begin{tikzcd}
				\Gamma && {\cU_i} \\
				& Y && X \\
				W && {W_i}
				\arrow["p"', from=1-1, to=3-1]
				\arrow["s", from=1-1, to=2-2]
				\arrow["{s_i}", from=1-3, to=2-4]
				\arrow["{p_i}"'{pos=0.2}, from=1-3, to=3-3]
				\arrow[from=1-1, to=1-3]
				\arrow["f"{pos=0.2}, from=2-2, to=2-4]
				\arrow[dashed, from=3-1, to=3-3]
				\arrow[dashed, from=2-2, to=3-1]
			\end{tikzcd}\]
			Let $E$ denote a general fiber $\Gamma_w$ such that $E$ is adjoint rigid with $a(E,L)=a(Y,L)$ and let $E^\prime$ be its image in $X$.
			We conclude that $({f}\circ s)|_E:E\rightarrow E^\prime$ is an adjoint rigid $a$-cover of an adjoint rigid subvariety $E^\prime$ of $X$ with $a(E^\prime,L)=a(X,L)$.
			
			It remains to have a case-by-case analysis for the family $W_i$.
			We may assume that the family is not contained in the union of singular fibers of $\pi_x$.
			As classified in Lemma \ref{lemm-a-cover-of-subvariety}, when $b(\overline{F},E^\prime,L)<b(F,X,L)=2$, the family $W_i$ parametrizes either
			\begin{itemize}
				\item fibers of $\pi_y$, or
				\item conics in fibers of $\pi_x$, or
				%					\item smooth $\pi_x$-fibers with Picard rank $1$.
			\end{itemize}
			Note that smooth $\pi_x$-fibers are not included since they have Picard rank $7$ over $\overline{F}$.
			In each case, there does not exist adjoint rigid $a$-cover of $E^\prime$ by Lemma \ref{lemm-a-cover-of-subvariety}, which implies that general fibers of $\Gamma\rightarrow\cU_i\rightarrow W_i$ are birational to disjoint copies of $E^\prime$.
			Thus we have that
			\begin{align*}
				b(F,Y,L)&=b(F,\Gamma,L)&\text{(since $Y$ is birational to $\Gamma$)}\\
				&\leq b(\overline{F},\Gamma,L)&\text{(by Proposition \ref{prop-properties-of-b})}\\
				&\leq b(\overline{F},E^\prime,L)&\text{(by Proposition \ref{prop-adjoint-rigid})}\\
				&< b(F,X,L)&\text{(by assumption)}\\
				&\leq b(F,Y,L),&\text{(by assumption)}
			\end{align*}
			which is a contradiction.
			%				$$b(F,Y,L)=b(F,\Gamma,L)\leq b(\overline{F},\Gamma,L)\leq b(\overline{F},E^\prime,L)< b(F,X,L)\leq b(F,Y,L)$$
			%				by Remark \ref{rema-b-is-birational}, Proposition \ref{prop-properties-of-b} and \ref{prop-adjoint-rigid}, a contradiction.
			
			Otherwise, we have that $b(\overline{F},E^\prime,L)\geq2$.
			The only possibility is when $W_i$ generically parametrizes smooth fibers of $\pi_x$.
			Let us show that $W_i$ is isomorphic to $\bP_x^3$.
			Since $\pi_x:X\ra \bP_x^3$ is a flat family generically parametrizes smooth fibers of $\pi_x$, there exists a morphism $u:\bP_x^3\ra W_i$ by the universal property of Hilbert schemes.
			Since each smooth fiber of $\pi_x$ appears in the family $\pi_x:X\ra \bP_x^3$ exactly once, the morphism $u$ is injective and dominant.
			Then $u$ is birational since we adopted the reduced structure for $W_i$.
			We conclude that $u$ is an isomorphism since $\bP_x^3$ is minimal.
			%				In this case, we have that $W_i\cong \bP_x^3$ and that the universal family is isomorphic to $\pi_x$.
			In this case, the diagram reads as follows.
			\begin{equation*}
				\begin{tikzcd}
					\Gamma && X \\
					& {{Y}} && X \\
					W && {\bP_x^3}
					\arrow["p"', from=1-1, to=3-1]
					\arrow["s", from=1-1, to=2-2]
					\arrow[dashed, from=2-2, to=3-1]
					\arrow[ Rightarrow,no head, from=1-3, to=2-4]
					\arrow["{\pi_x}"'{pos=0.2}, from=1-3, to=3-3]
					\arrow["h",from=1-1, to=1-3]
					\arrow["{{f}}"{pos=0.2}, from=2-2, to=2-4]
					\arrow["g",dashed,from=3-1, to=3-3]
				\end{tikzcd}
			\end{equation*}
			
			%			By taking the Stein factorization $\Gamma\rightarrow W^\prime\rightarrow \bP_x^3$ of $\Gamma\rightarrow \bP_x^3$ and replacing $W$ by $W^\prime$,
			%			we may assume that $p$ has connected fibers.
			%			Then $W$ is reduced and irreducible by \cite[Tags \href{https://stacks.math.columbia.edu/tag/0AXN}{0AXN} and \href{https://stacks.math.columbia.edu/tag/0AXP}{0AXP}]{stacks-project}.
			Let $W^\circ:=W\backslash g^{-1}(V_0)$ denote the open subscheme corresponding to smooth fibers of $\pi_x$.
			Since $W$ and $X$ are irreducible, the dimension of the image of $h_w:\Gamma_w\rightarrow X_a$ is constant for general $w\in W$ by  \cite[\href{https://stacks.math.columbia.edu/tag/05F7}{Tag 05F7}]{stacks-project}, where $a:=g(w)$.
			We may assume that $h_w:\Gamma_w\rightarrow X_a$ is dominant for general $w\in W$, since otherwise $\Gamma\rightarrow W$ will factor through a family of dimension $1$ subvarieties, which has been discussed above.
			Then the conclusion follows by Theorem \ref{theo-monodromy-final}.
		\end{proof}

		\begin{proof}[Proof of Corollary \ref{coro-thin-set}]
			The set $Z$ is a thin set since it is defined by finitely many proper closed subsets $V_\tau$ and finitely many finite morphisms $f_\tau$.
			Each element of $V_\tau(F)$ lies on a line in a $\pi_x$-fiber, which has $a$-value equals $2$;
			each element of $f_\tau(X_\tau(F))$ lies either on a singular fiber of $\pi_x$, which is already contained in $V_\tau$, or on a smooth fiber $S$ of $\pi_x$ satisfying $a(S,L)=1$ and $b(F,S,L)\geq2$.
			Thus $Z\subseteq \bigcup_f f(Y(F))$ in equation (\ref{eq}).
			Conversely, let $f:Y\rightarrow X$ be an $F$-thin map where $Y$ is a smooth geometrically integral variety satisfying
			$$(a(Y,f^\ast L),b(F,Y,f^\ast L))\geq(a(X,L),b(F,X,L))=(1,2)$$
			in lexicographical order.
			By Theorem \ref{theo-factor-through}, there exists a rational map $\phi:Y\dasharrow Q$ such that the following diagram commutes:
			% https://q.uiver.app/#q=WzAsMyxbMCwwLCJZIl0sWzAsMSwiUSJdLFsxLDEsIlgiXSxbMCwxLCJcXHBoaSIsMix7InN0eWxlIjp7ImJvZHkiOnsibmFtZSI6ImRhc2hlZCJ9fX1dLFswLDIsImYiXSxbMSwyLCJxIiwyXV0=
			\[\begin{tikzcd}
				Y \\
				Q & X
				\arrow["\phi"', dashed, from=1-1, to=2-1]
				\arrow["f", from=1-1, to=2-2]
				\arrow["q"', from=2-1, to=2-2]
			\end{tikzcd}\]
			where $q:Q\ra X$ is either the inclusion of $V_\tau\subset X$ or the morphism $f_\tau:X_\tau\ra X$, for some $\tau\in\mathfrak{S}_4$.
			Since $Y$ is smooth, the Lang-Nishimura theorem implies that $f(Y(F))\subseteq q(Q(F))\subseteq Z$.
			Thus, equation (\ref{eq}) holds.

			By Theorem \ref{theo-a-cover}, there is no thin map $f:Y\ra X$ satisfying $\dim Y=\dim X$ and $\kappa(Y,K_Y+a(Y,f^\ast L)f^\ast L)=0$.
			Thus, removing restrictions (a), (b) and (c) in Definition \ref{defi-ges} does not affect the geometric exceptional set of $(X,L)$.
			Hence the last assertion of Corollary \ref{coro-thin-set} holds.
			%	in the notation of Definition \ref{defi-ges}, we have that $\dim Y=\dim X$ implies $\kappa(Y,K_Y+a(Y,f^\ast L)f^\ast L)>0$ by Theorem \ref{theo-a-cover}.
		\end{proof}

%\nocite{*}
\normalem%用于解决\usepackage{ulem}的bug
\bibliographystyle{alpha}
\bibliography{myref}

%\printbibliography

\end{document}